\documentclass{amsart}
\usepackage{fullpage}
\usepackage{amsmath} \usepackage{amssymb} \usepackage{amsopn}
\usepackage{amsthm} \usepackage{amsfonts} \usepackage{mathrsfs}
\usepackage{mathabx}

\usepackage{verbatim}

\usepackage[OT2,T1]{fontenc}

\usepackage[pdfpagelabels,pdftex]{hyperref}
\usepackage[all]{xy}

\usepackage{amssymb}
\usepackage{euscript}
\usepackage{cite}
\usepackage[all]{xy}
\usepackage{tabularx}

\theoremstyle{plain}
\usepackage{amsfonts}
\usepackage{graphicx}
\usepackage{amsmath}

\hypersetup{
  pdftitle={},
  pdfauthor={},
  pdfsubject={},
  pdfkeywords={},
  colorlinks=false,
  breaklinks=true,
  bookmarksopen=true,
  bookmarksnumbered=true,
  pdfpagemode=UseOutlines,
  plainpages=false}

\newcommand{\bC}{{\mathbb C}}

\newcommand{\bF}{{\mathbb F}}

\newcommand{\bN}{{\mathbb N}}

\newcommand{\bQ}{{\mathbb Q}}
\newcommand{\bR}{{\mathbb R}}

\newcommand{\bZ}{{\mathbb Z}}

\newcommand{\cA}{{\mathscr A}}
\newcommand{\cB}{{\mathscr B}}

\newcommand{\cG}{{\mathscr G}}

\newcommand{\cI}{{\mathscr I}}

\newcommand{\cL}{{\mathscr L}}
\newcommand{\cM}{{\mathscr M}}
\newcommand{\cN}{{\mathscr N}}

\newcommand{\cP}{{\mathscr P}}

\newcommand{\cS}{{\mathscr S}}

\newcommand{\caO}{{\mathcal O}}

\newcommand{\fP}{{\mathfrak P}}
\newcommand{\fQ}{{\mathfrak Q}}

\newcommand{\fc}{{\mathfrak c}}

\newcommand{\fl}{{\mathfrak l}}

\newcommand{\fp}{{\mathfrak p}}
\newcommand{\fq}{{\mathfrak q}}

\DeclareSymbolFont{cyrletters}{OT2}{wncyr}{m}{n}
\DeclareMathSymbol{\Sha}{\mathalpha}{cyrletters}{"58}

\DeclareMathOperator{\id}{id}

\DeclareMathOperator{\Out}{Out}
\DeclareMathOperator{\Hom}{Hom}

\DeclareMathOperator{\Aut}{Aut}

\DeclareMathOperator{\coker}{coker}
\DeclareMathOperator{\im}{im}

\DeclareMathOperator{\GL}{GL}

\DeclareMathOperator{\Spec}{Spec}

\DeclareMathOperator{\R}{R}

\DeclareMathOperator{\res}{res}
\DeclareMathOperator{\cores}{cor}

\newcommand{\cd}{{\rm cd}}
\newcommand{\scd}{{\rm scd}}

\DeclareMathOperator{\Gal}{G}

\newcommand{\et}{\text{\rm \'et}}

\newcommand{\Fet}{{\rm Fet}}

\newcommand{\sep}{{\rm sep}}

\newcommand{\nr}{{\rm nr}}

\newcommand{\ab}{{\rm ab}}

\makeatletter
\newtheorem*{rep@theorem}{\rep@title}
\newcommand{\newreptheorem}[2]{%
\newenvironment{rep#1}[1]{%
 \def\rep@title{#2 \ref{##1}}%
 \begin{rep@theorem}}%
 {\end{rep@theorem}}}
\makeatother

\newtheorem{thm}{Theorem}[section]
\newtheorem*{introThm}{Theorem}

\newtheorem{prop}[thm]{Proposition}

\newtheorem{cor}[thm]{Corollary}

\newtheorem{lm}[thm]{Lemma}

\newtheorem{conj}[thm]{Conjecture}

\newtheorem{conv}[thm]{Convention}

\newreptheorem{thm}{Theorem}
\newreptheorem{prop}{Proposition}
\newreptheorem{cor}{Corollary}

\theoremstyle{definition}
\newtheorem{Def}[thm]{Definition}

\newtheorem{rem}[thm]{Remark}
\newtheorem{rems}[thm]{Remarks}
\newtheorem{ex}[thm]{Example}

\newenvironment{pro*}[1][Proof]{{\it{#1:}} }{}

\newcommand{\abs}{\sharp}

\newcommand\cool{\overline{\mathbb{Q}_{\ell}}}

\newcommand\rar{ \rightarrow }
\newcommand\tar{ \twoheadrightarrow }
\newcommand\har{ \hookrightarrow }
\newcommand\Rar{ \Rightarrow }
\newcommand\LRar{ \Leftrightarrow }
\newcommand\longrar{\longrightarrow}

\newcommand\cs{\mathop{ \rm cs}}

\newcommand\dirlim{\mathop{\underrightarrow{\lim} }}
\newcommand\prolim{\mathop{\underleftarrow{\lim} }}
\newcommand{\sm}{{\,\smallsetminus\,}}

\newcommand\coh{\rm H}

\newcommand\subsetsim{\stackrel{\subset}{\sim}}
\newcommand\supsetsim{\stackrel{\supset}{\sim}}

\newcommand\Ind{\mathop{ \rm Ind}}
\DeclareMathOperator{\Resdirsum}{\bigoplus\nolimits^{\prime}}

\newcommand\Kapi{\mathop{ \rm K(\pi,1)}}
\DeclareMathOperator{\Ram}{Ram}

\newcounter{absatzcounter}[section]
\numberwithin{equation}{section}

\begin{document}

\title{Stable sets of primes in number fields}
\author{A. Ivanov}
\email{ivanov@mathi.uni-heidelberg.de}
\address{Mathematisches Institut, Universit\"at Heidelberg, Im Neuenheimer Feld 288, 69120 Heidelberg, Germany}
\thanks{The author was supported by the Mathematical Center Heidelberg}
\begin{abstract}
We define a new class of sets -- stable sets -- of primes in number
fields. For example, Chebotarev sets $P_{M/K}(\sigma)$, with $M/K$ Galois
and $\sigma \in \Gal(M/K)$, are very often stable. These sets have
positive (but arbitrarily small) Dirichlet density and they generalize sets with
density one in the sense that arithmetic theorems like certain Hasse
principles, Grunwald-Wang theorem, Riemann's existence theorem,
etc. hold for them. Geometrically this allows to give examples of infinite
sets $S$ with arbitrarily small positive density such that $\Spec
\mathcal{O}_{K,S}$ is a $K(\pi,1)$ (simultaneously for all $p$).
\end{abstract}

\maketitle
\setcounter{tocdepth}{1}
\tableofcontents

\section{Introduction}

The main goal of this article is to define a new class of sets of primes of positive Dirichlet density in number fields -- stable sets. These sets have a positive but arbitrarily small density and they generalize in many aspects sets of density one. In particular, most of the arithmetic theorems, such as certain Hasse principles, Grunwald-Wang theorem, Riemann's existence theorem, $\Kapi$-property, etc., which hold for sets of density one (cf. \cite{NSW} Chapters IX and X), also hold for stable sets. Our goals are on the one hand to prove these arithmetic results, and on the other hand to give many examples of stable sets.

The idea is as follows: let $\lambda > 1$. A set $S$ of primes in a number field $K$ is $\lambda$-\emph{stable} for the extension $\cL/K$, if there is a subset $S_0 \subseteq S$, a finite subextension $\cL/L_0/K$ and some $a > 0$ such that we have $\delta_L(S_0) \in [a,\lambda a)$ for all finite subextensions $\cL/L/L_0$, where $\delta_L$ is the Dirichlet density. We call the field $L_0$ a \emph{$\lambda$-stabilizing field} for $S$ for $\cL/K$. A more restrictive version is the notion of persistent sets: $S$ is \emph{persistent} if the function $L \mapsto \delta_{L}(S_0)$ gets constant in the tower $\cL/K$ beginning from some finite subextension $L_0/K$ (cf. Definition \ref{def:stable_and_persisting_sets}). In particular, for any $\lambda > 1$, a $\lambda$-stable set is persistent.

The main result in this article is the following theorem, which links stability to vanishing of certain Shafarevich-Tate groups. Let $\Sha^1$ denote the usual Shafarevich-Tate group, consisting of global cohomology classes, which vanish locally in a given set of primes and if $A$ is a module over a finite group $G$, then $\coh^1_{\ast}(G,A)$ denotes the subgroup of $\coh^1(G,A)$ consisting of precisely those classes, which vanish after restriction to all cyclic subgroups of $G$. Moreover, if $\cL/L$ is a Galois extension of fields, then $\Gal_{\cL/L}$ denotes its Galois group, and if $A$ is a $\Gal_{\cL/L}$-module, then $L(A)/L$ denotes the trivializing extension of $A$.

\begin{repthm}{thm:keyresult}
Let $K$ be a number field, $T$ a set of primes of $K$ and $\cL/K$ a Galois extension. Let $A$ be a finite $G_{\cL/K}$-module. Assume that $T$ is $p$-stable for $\cL/K$, where $p$ is the smallest prime divisor of $\abs{A}$. Let $L$ be a $p$-stabilizing field for $T$ for $\cL/K$. Then:

\[ \Sha^1(\cL/L,T; A) \subseteq \coh^1_{\ast}(L(A)/L, A). \]

\noindent In particular, if $\coh^1_{\ast}(L(A)/L, A) = 0$, then $\Sha^1(\cL/L,T; A) = 0$.
\end{repthm}

This theorem has many applications to the structure of the Galois group $\Gal_{K,S} := {\rm Gal}(K_S/K)$, where $K$ is a number field and $S$ is stable. To explain our results, we need some notations. If $S, R$ are two sets of primes of a number field $K$, then we denote by $K_S^R$ the maximal extension of $K$, which is unramified outside $S$ and completely split in $R$. Moreover, we denote by $\Gal_{K,S}^R$ the Galois group of $K_S^R/K$. Let $\cL/K$ be any Galois extension. For a prime $\fp$ of $K$ we denote by $\cL_{\fp}$ the completion of $\cL$ at a (any) extension of $\fp$ to $\cL$ (the isomorphism class of the completion $\cL_{\fp}$ does not depend on the particular choice of the extension of $\fp$ to $\cL$ as $\cL/K$ is Galois, and we suppress this choice in our notation). Furthermore, $\cG_{\fp}$ denotes the absolute Galois group of $K_{\fp}$, and $K_{\fp}(p)$ (resp. $K_{\fp}^{\nr}(p)$) denotes the maximal (resp. the maximal unramified) pro-p extension of $K_{\fp}$. Moreover, for a pro-finite group $G$, we denote the pro-$p$-completion of $G$ by $G(p)$. For more notations, see also the end of this introduction.

\begin{introThm}(cf. Theorems \ref{thm:LocExt_RExt_CD} and \ref{thm:Kapi1forstable})
Let $K$ be a number field, $p$ a rational prime, $\fp$ a prime of $K$ and $T \supseteq S \supseteq R$ sets of primes of $K$ with $R$ finite. Assume that $S$ is $p$-stable\footnote{In fact a weaker condition would suffice, cf. Theorem \ref{thm:LocExt_RExt_CD}.} for $K_S^R(\mu_p)/K$. Then
\begin{itemize}
\item[(A)](Local extensions)   \[ K_{S,\fp}^R \supseteq \begin{cases} K_{\fp}(p) & \text{if } \fp \in S \sm R \\ K_{\fp}^{\nr}(p) & \text{if } \fp \not\in S. \end{cases} \]
\item[(B)](Riemann's existence theorem) Let $I_{\fp}^{\prime}(p)$ denote the Galois group of the maximal pro-$p$ extension of $K_{S,\fp}^R$ and $K_T^{\prime}(p)/K_S^R$ denote the maximal pro-$p$ subextension of $K_T/K_S^R$. The natural map
\[ \phi_{T,S}^R \colon \bigast\limits_{\fp \in R(K_S^R)} \cG_{\fp}(p) \ast \bigast\limits_{\fp \in (T \sm S)(K_S^R)} I_{\fp}^{\prime}(p) \stackrel{\sim}{\longrar} \Gal_{K_T^{\prime}(p)/K_S^R} \]
\noindent is an isomorphism (where $\bigast$ is to be understood in the sense of \cite{NSW} Chapter IV).
\item[(C)](Cohomological dimension) Assume that either $p$ is odd, or $K$ is totally imaginary. Then 
\[ \cd_p \Gal_{K,S}^R = \scd_p \Gal_{K,S}^R = 2. \]
\item[(D)] ($\Kapi$-property) Assume additionally that $R = \emptyset$, $S \supseteq S_{\infty}$ and that either $p$ is odd, or $K$ is totally imaginary. Then $\Spec \caO_{K,S}$ is a $\Kapi$ for $p$ (cf. Definition \ref{def:Kapidef}).
\end{itemize}
\end{introThm}

There are also corresponding results for the maximal pro-$p$ quotient $\Gal_{K,S}^R(p)$ of $\Gal_{K,S}^R$. These results are essentially well-known (cf. \cite{NSW}) if $\delta_K(S) = 1$ resp. if $S \supseteq S_p \cup S_{\infty}$. Also, A. Schmidt showed recently that if $T_0$ is any fixed set with $\delta_K(T_0) = 1$ and $S$ is arbitrary finite set of primes, then there is a finite subset $T_1 \subseteq T_0$ (depending on $S$) such that the pro-$p$ versions of the above results essentially (e.g., except the result on $\scd_p$) hold if one replaces $S$ by $S \cup T_1$ (cf. \cite{Sch},\cite{Sch2},\cite{Sch3}). 

A further application of stable sets concerns a generalization of the Neukirch-Uchida theorem, which is a result of anabelian nature. More details on this can be found in \cite{Iv} Section 6. Now we see many examples of stable (even persistent) sets:

\begin{repcor}{cor:stability_of_ACS}
Let $M/K$ be finite Galois and let $\sigma \in \Gal_{M/K}$. Let $S \backsimeq P_{M/K}(\sigma)$ (i.e., up to a density zero subset, $S$ is equal to $P_{M/K}(\sigma)$). Let $\cL/K$ be any extension. Then $S$ is persistent (or, equivalently, stable; cf. Corollary \ref{prop:stabioity_properties_of_ACS}) for $\cL/K$ if and only if
\[ \Gal_{M/M \cap \cL} \cap C(\sigma; \Gal_{M/K}) \neq \emptyset, \]
where $C(\sigma;\Gal_{M/K})$ denotes the conjugacy class of $\sigma$ in $\Gal_{M/K}$. In particular,
\begin{itemize}
\item[(i)]  If $\sigma = 1$, then $S \backsimeq P_{M/K}(1) = \cs(M/K)$ is persistent for any extension $\cL/K$.
\item[(ii)] If $M \cap \cL = K$, then $S \backsimeq P_{M/K}(\sigma)$ is persistent for $\cL/K$.
\end{itemize}
\end{repcor}


\subsection*{Outline of the article}

In Section \ref{sec:s_and_p_sets} we introduce stable, sharply $p$-stable, strongly $p$-stable and persistent sets. Section \ref{sec:Examples_section} is devoted to examples: in particular, we introduce \emph{almost Chebotarev sets}, which provide us with a rich supply of persistent sets (Section \ref{sec:ACS}), and we show that essentially an almost Chebotarev set is sharply and strongly $p$-stable for almost all $p$ (Section \ref{sec:ex_propertiesast}). In Section \ref{sec:stablesets_and_Sha1_key_result} we prove our main result which is a general Hasse principle. In Sections \ref{sec:concrete_HPs_for_stablesets}-\ref{sec:unif_bounds_on_Sha1_field_var} we discuss some further Hasse principles and uniform bounds on Shafarevich-Tate groups for stable sets. In Section \ref{sec:Arith_appl} we deduce arithmetic applications such as the Grunwald-Wang theorem, realization of local extensions, Riemann's existence theorem and 
cohomological dimension. In Section \ref{sec:Kapi_section} we deduce the $\Kapi$ property at $p$ for $\Spec \caO_{K,S}$ with $S$ being sharply $p$-stable using results from Section \ref{sec:Arith_appl}.


\subsection*{Notation}

Our notation essentially coincides with the notation in \cite{NSW}. We collect some of the most important notations here. For a pro-finite group $G$ we denote by $G(p)$ its maximal pro-$p$ quotient. For a subgroup $H \subseteq G$, we denote by $N_G(H)$ its normalizer in $G$. If $\sigma \in G$, then we write $C(\sigma;G)$ for its conjugacy class. For two finite groups $H \subseteq G$, we write $m_H^G$ (resp. $m_H$, if $G$ is clear from the context) for the character of the induced representation $\Ind_H^G \bold{1}_H$.

For a Galois extension $M/L$ of fields, $\Gal_{M/L}$ denotes its Galois group and $L(p)$ denotes the maximal pro-$p$ extension of $L$ (in a fixed algebraic closure). By $K$ we always denote an algebraic number field, that is a finite extension of $\bQ$. If $\fp$ is a prime of $K$ and $L/K$ is a Galois extension, then $D_{\fp, L/K} \subseteq \Gal_{L/K}$ denotes the decomposition subgroup of $\fp$. We write $\Sigma_K$ for the set of all primes of $K$ and $S,T,R, \dots$ will usually denote subsets of $\Sigma_K$. If $L/K$ is an extension and $S$ a set of primes of $K$, then we denote the pull-back of $S$ to $L$ by $S_L$, $S(L)$ or $S$ (if no ambiguity can occur). We write $K_S^R/K$ for the maximal extension of $K$, which is unramified outside $S$ and completely split in $R$ and $\Gal_S^R := \Gal_{K,S}^R$ for its Galois group. We use the abbreviations $K_S := K_S^{\emptyset}$ and $\Gal_S := \Gal_S^{\emptyset}$. Further, for $p \leq \infty$ a (archimedean or non-archimedean) prime of $\bQ$, $S_p = S_p(K)$ denotes the 
set of all primes of $K$ lying over $p$. Further, if $S \subseteq \Sigma_K$, we write $\bN(S) := \bN \cap \caO_{K,S}^{\ast}$, i.e., $p \in \bN(S)$ if and only if $S_p \subseteq S$.

We write $\delta_K$ for the Dirichlet density on $\Sigma_K$. For $S, T$ subsets of $\Sigma_K$, we use (following \cite{NSW} 9.1.2)
\begin{eqnarray*}
S \subsetsim T &:\LRar& \delta_K( S \sm T) = 0 \\
S \backsimeq T &:\LRar& (S \subsetsim T) \text{ and } (T \subsetsim S).
\end{eqnarray*}

\noindent Thus $S \subsetsim T$ if $S$ is contained in $T$ up to a set of primes of density zero. For a finite Galois extension $M/K$ and $\sigma \in \Gal_{M/K}$, we have the Chebotarev set
\[ P_{M/K}(\sigma) = \{ \fp \in \Sigma_K \colon \fp \text{ is unramified in $M/K$ and } (\fp, M/K) = C(\sigma ; \Gal_{M/K}) \}, \]
\noindent where $(\fp, M/K)$ denotes the conjugacy class of Frobenius elements corresponding to primes of $M$ lying over $\fp$.


\subsection*{Acknowledgements}
A part of the results in this article coincide with the results in author's Ph.D. thesis \cite{Iv}, which was written under the supervision of Jakob Stix at the University of Heidelberg. The author is very grateful to him for the very good supervision, and to Kay Wingberg, Johannes Schmidt and a lot of other people for very helpful remarks and interesting discussions. The work on the author's thesis was partially supported by Mathematical Center Heidelberg and the Mathematical Institute Heidelberg. The author thanks both institutions for their hospitality and the excellent working conditions.


\section{Stable and persistent sets} \label{sec:s_and_p_sets}

\subsection{Warm-up: preliminaries on Dirichlet density} \label{sec:warm_up}
Let $\cP_K$ denote the set of all subsets of $\Sigma_K$. The Dirichlet density $\delta_K$ is not defined for all elements in $\cP_K$. Moreover, there are examples of finite extensions $L/K$ and $S \in \cP_K$, such that $S$ has a density, but the pull-back $S_L$ of $S$ to $L$ has no density. To avoid dealing with such sets we make the following convention, which holds until the end of this article.

\begin{conv} \label{conv:density_ex}
If $S \in \cP_K$ is a set of primes of $K$, then we assume implicitly that for all finite extensions $L/K$, all finite Galois extensions $M/L$ and all $\sigma \in \Gal_{M/L}$, the set $S_L \cap P_{M/L}(\sigma)$ has a Dirichlet density.\footnote{The optimal way to omit sets having no density would be to find an appropriate sub-$\sigma$-algebra of $\cP_K$ (for any $K$), such that the restriction of $\delta_K$ to it is a measure (and the pull-back maps $\cP_K \rar \cP_L$ attached to finite extensions $L/K$ restrict to pull-back maps on these sub-$\sigma$-algebras). Unfortunately, there is no satisfactory way to find such $\sigma$-algebra $\cB_K$, at least if one requires that if $S \in \cB_K$, then also $T \in \cB_K$ for any $T \backsimeq S$, or, which is weaker, that any finite set of primes of $K$ lies in $\cB_K$. Indeed, countability of $\Sigma_K$ would imply $\cB_K = \cP_K$ in this case, but not all elements of $\cP_K$ have a Dirichlet density.}
\end{conv}

Convention \ref{conv:density_ex} is satisfied for all sets lying in the following rather big subset of $\cP_K$:

\[ \cA_K := \left\{  \begin{array}{cl} S \subseteq \Sigma_K \colon S \backsimeq \bigcup_i P_{L_i/K_i}(\sigma_i)_K \text{ for some } K/K_i/\bQ \\ \text{and } L_i/K_i \text{ finite Galois and }\sigma_i \in \Gal_{L_i/K_i} \end{array} \right\}, \]

\noindent where the unions are \noindent{disjoint} and countable (or finite or empty). This set $\cA_K$ can not be closed simultaneously under (arbitrary) unions and complements: otherwise it would be a $\sigma$-algebra and hence would be equal to $\cP_K$.

To compute the density of pull-backs of sets we use the following two lemmas. Let $L/K$ be a finite extension of degree $n$ (not necessarily Galois). For $0 \leq m \leq n$, define the following sets:
\[ P_m(L/K) := \{\fp \in \Sigma_K \colon \fp \text{ is unramified and has exactly } m \text{ degree-1-factors in } L \}. \]

\noindent In particular, $P_n(L/K) = \cs(L/K)$, $P_{n-1}(L/K) = \emptyset$. Recall that if $H \subseteq G$ are finite groups, then $m_H$ denotes the character of the $G$-representation $\Ind_H^G \bold{1}$. One has:

\[ m_H(\sigma) = \abs{ \{ gH \colon \langle \sigma \rangle^g \subseteq H  \} } = \abs{ \{ \langle \sigma \rangle g H \colon \langle \sigma \rangle^g \subseteq H \} }, \]

\noindent where $\langle \sigma \rangle \subseteq G$ denotes the subgroup generated by $\sigma$ and $\langle \sigma \rangle^g := g^{-1} \langle \sigma \rangle g$. The second equality follows immediately from the fact that if $\langle \sigma \rangle^g \subseteq H$, then $gH = \langle \sigma \rangle g H$.

\begin{lm} \label{lm:P_i_setsZerlegung}
Let $L/K$ be a finite extension and $N/K$ a finite Galois extension containing $L$, with Galois group $G$, such that $L$ corresponds to a subgroup $H \subseteq G$. Then

\[ P_m(L/K) \backsimeq \{ \fp \in P_m(L/K) \colon \fp \text{ is unramified in } N/K \} = \bigcup_{ \stackrel{C(\sigma;G) \subseteq G}{m_H(\sigma)  = m } } P_{N/K}(\sigma) \]

\noindent (disjoint union). In particular, $P_m(L/K) \in \cA_K$ and
\[ \delta_K(P_m(L/K)) = \abs{G}^{-1} \sum_{ \stackrel{C(\sigma;G) \subseteq G}{m_H(\sigma) = m} } \abs{C(\sigma;G)} .  \]
\end{lm}

\begin{proof} The proof of the first statement is an elementary exercise in Galois theory (if $\fp$ is a prime of $K$ unramified in $N$, then the primes of $L$ lying over $\fp$ are in one-to-one correspondence with double cosets $\langle \sigma \rangle g H$, where $\sigma$ is arbitrary in the Frobenius class of $\fp$; the residue field extension of a prime belonging to the coset $\langle \sigma \rangle g H$ over $\fp$ has the Galois group $\langle \sigma \rangle^g/\langle \sigma \rangle^g \cap H$). The second statement follows from the first and the Chebotarev density theorem.
\end{proof}

\begin{lm}\label{lm:Dichtepullback}
Let $L/K$ be a finite extension of degree $n$, $S$ a set of primes of $K$ and $N/K$ a Galois extension containing $L$, such that $G := \Gal_{N/K} \supseteq \Gal_{N/L} =: H$. Then
\[ \delta_L(S) = \sum_{m=1}^n m \delta_K(S \cap P_m(L/K)) = \sum_{C(\sigma ; G) \subseteq G } m_H(\sigma) \delta_K(S \cap P_{N/K}(\sigma)). \]
\noindent If, in particular, $L/K$ is Galois, we get the well-known formula $\delta_L(S) = [L:K] \delta_K(S \cap \cs(L/K))$.
\end{lm}

\begin{proof}
The first equation is an easy computation and the second follows from Lemma \ref{lm:P_i_setsZerlegung}.
\end{proof}


\subsection{Key definitions} \label{sec:def_of_s_and_p_sets}

Let $K$ be a number field and $S$ a set of primes. If $\delta_K(S) = 0$ resp. $= 1$, then also $\delta_L(S) = 0$ resp. $= 1$ for all finite $L/K$. Now, if $0 < \delta_K(S) < 1$, then it can happen that there is some finite $L/K$ with $\delta_L(S) = 0$ (take a finite Galois extension $L/K$ and set $S := \Sigma_K \sm \cs(L/K)$; then $\delta_K(S) = 1 - [L:K]^{-1}$ and $\delta_L(S_L) = 0$). For stable sets, defined below, this possibility is excluded.

\begin{Def}\label{def:stable_and_persisting_sets}
Let $S$ be a set of primes of $K$, $\cL/K$ any extension and $\lambda > 1$. A finite subextension $\cL/L_0/K$ is called \textbf{$\lambda$-stabilizing} for $S$ for $\cL/K$, if there exists a subset $S_0 \subseteq S$ and some $a \in (0,1]$, such that $\lambda a > \delta_L(S_0) \geq a > 0$ for all finite subextensions $\cL/L/L_0$. Moreover, we call $L_0$ \textbf{persisting} for $S$ for $\cL/K$, if there exists a subset $S_0 \subseteq S$, such that $\delta_L(S_0) = \delta_{L_0}(S_0) > 0$ for all finite subextensions $\cL/L/L_0$. Further,

\begin{itemize}
\item[(i)]  We call $S$ \textbf{$\lambda$-stable} (resp. \textbf{persistent}) for $\cL/K$ if it has a  $\lambda$-stabilizing resp. persisting extension for $\cL/K$.
\item[(ii)] We call $S$ \textbf{stable} for $\cL/K$, if there is a $\lambda > 1$ such that $S$ is $\lambda$-stable for $\cL/K$.
\end{itemize}

\noindent Assume that $\lambda = p$ is a rational prime.
\begin{itemize}
\item[(iii)]  We call $S$ \textbf{sharply $p$-stable} for $\cL/K$ if $\mu_p \subseteq \cL$ and $S$ is $p$-stable for $\cL/K$, or $\mu_p \not\subseteq \cL$ and $S$ is stable for $\cL(\mu_p)/K$.
\end{itemize}
\end{Def}

In applications in this article we will often use the case $\cL = K_S$. Therefore, we define:

\begin{Def}\label{def:stable_and_pers_sets_wrt_K_S}
Let $S$ be a set of primes of $K$ and $\lambda > 1$. 
\begin{itemize}
\item[(i)] We call $S$ $\lambda$-\textbf{stable} (resp. \textbf{stable}, resp. \textbf{presistent}), if $S$ is $\lambda$-stable (resp. stable, resp. persistent) for $K_S/K$.
\end{itemize}
\noindent Assume that $\lambda = p$ is a rational prime.
\begin{itemize}
\item[(ii)] We call $S$ \textbf{sharply $p$-stable} if $S$ is sharply $p$-stable for $K_S/K$. Moreover, we define the exceptional set $E^{\rm sharp}(S)$ to be the set of all rational primes $p$ such that $S$ is not sharply $p$-stable.
\item[(iii)] We call $S$ \textbf{strongly $p$-stable}, if $S$ is $p$-stable for $K_{S \cup S_p \cup S_{\infty}}/K$ with a $p$-stabilizing field contained in $K_S$. Further, we call $S$ strongly $\infty$-stable, if $S$ is stable for $K_{S \cup S_{\infty}}/K$. Moreover, we define the exceptional set $E^{\rm strong}(S)$ to be the set of all rational primes $p$ or $p = \infty$, such that $S$ is not strongly $p$-stable.
\end{itemize}
\end{Def}

Clearly, a strongly $p$-stable set is also $p$-stable and sharply $p$-stable. In particular, we have $E^{\rm strong}(S) \supseteq E^{\rm sharp}(S)$. On the other side, in general, neither one of the properties '$p$-stable' and 'sharply $p$-stable' implies the other. 

\begin{lm}\label{prop:firstpropsofstabletrivs} Let $\cL/K$ be an extension and $S$ a set of primes of $K$.
\begin{itemize}
\item[(i)] Let $\lambda \geq \mu > 1$.  If $S$ is $\mu$-stable with $\mu$-stabilizing field $L_0$, then $S$ is $\lambda$-stable with $\lambda$-stabilizing field $L_0$.
\item[(ii)]  If $L_0$ is $\lambda$-stabilizing resp. persisting field for $S$ for $\cL/K$, then any finite subextension $\cL/L_1/L_0$ has the same property.

\item[(iii)] Let $S^{\prime}$ be a further set of primes of $K$. If $S \subsetsim S^{\prime}$, and $S$ is $\lambda$-stable resp. persistent for $\cL/K$, then  $S^{\prime}$ also has this property. Any $\lambda$-stabilizing resp. persisting field for $S$ has the same property for $S^{\prime}$.

\item[(iv)]   Let $\cL/\cN/M/K$ be subextensions. If $S$ is $\lambda$-stable resp. persistent for $\cL/K$ with $\lambda$-stabilizing resp. persisting field $L_0 \subseteq \cN$, then $S_M$ is $\lambda$-stable resp. persistent for $\cN/M$.
\end{itemize}
\end{lm}

\begin{lm} \label{lm:dagger_rel_p_is_nice}
Let $\cL/K$ be an extension and $S$ a set of primes of $K$. Assume that $S$ is sharply $p$-stable for $\cL/K$. There is a finite subextension $\cL/L_0/K$, such that for any subextensions $\cL/\cN/L/L_0$ (with $L/L_0$ finite) $S$ is sharply $p$-stable for $\cN/L)$.
\end{lm}

The proofs of these lemmas are straightforward. The following proposition gives another characterization of stable sets and shows in particular, that if $S$ is stable for $\cL/K$, then any finite subfield $\cL/L/K$ is $\lambda$-stabilizing for $S$ with a certain $\lambda > 1$ depending on $L$.

\begin{prop}\label{prop:stableequalbounded} Let $S$ be a set of primes of $K$ and $\cL/K$ an extension. The following are equivalent:
\begin{itemize}
\item[(i)]   $S$ is stable for $\cL/K$.
\item[(ii)]  There exists some $\lambda > 1$, such that $S$ is $\lambda$-stable for $\cL/K$ with $\lambda$-stabilizing field $K$.
\item[(iii)] There exist some $\epsilon > 0$ such that $\delta_L(S) > \epsilon$ for all finite $\cL/L/K$.
\end{itemize}
\end{prop}

\begin{proof}
(iii) $\Rar$ (ii) $\Rar$ (i) are trivial. We prove (i) $\Rar$ (iii). Let $\lambda > 1$ and let $S$ be $\lambda$-stable for $\cL/K$ with $\lambda$-stabilizing field $L_0$. Then there is some $a > 0$ and a subset $S_0 \subseteq S$ such that $a \leq \delta_L(S_0) < \lambda a$ for all $\cL/L/L_0$. Suppose there is no $\epsilon > 0$, such that $\delta_{L}(S_0) > \epsilon$ for all $\cL/L/K$. This implies that there is a family $(M_i)_{i = 1}^{\infty}$ of finite subextensions of $\cL/K$ with $\delta_{M_i}(S_0) \rar 0$ as $i \rar \infty$. Then $d_i = [L_0 M_i : M_i] = [L_0 : L_0 \cap M_i]$ is bounded from above by $[L_0 : K]$ and hence by Lemma \ref{lm:Dichtepullback}

\[ \delta_{L_0 M_i}(S_0) = \sum_{m=1}^{d_i} m\delta_{M_i}(S_0 \cap P_m(L_0 M_i / M_i)) \leq [L_0 : K] \delta_{M_i}(S_0) \rar 0 \]

\noindent for $i \rar \infty$. This contradicts to the $\lambda$-stability of $S_0$ with respect to the $\lambda$-stabilizing field $L_0$.
\end{proof}

Here is a small overview of the use of these conditions and the examples in the practice:

\begin{itemize}
\item[-] The most examples of stable sets are given by (almost) Chebotarev sets, i.e. sets of the form $S \backsimeq P_{M/K}(\sigma)$, or sets containing them (cf. Section \ref{sec:ACS}).
\item[-] If $S$ is stable for $\cL/K$, then $\delta_L(S) > 0$ for all finite $\cL/L/K$. The converse is not true in general (cf. \cite{Iv} Section 3.5.4), but it is true for almost Chebotarev sets (cf. Section \ref{sec:ACS}).
\item[-] If an almost Chebotarev set is stable for an extension, then it is also persistent for it (cf. Corollary \ref{prop:stabioity_properties_of_ACS}). It is not clear whether there are examples of stable but not persistent sets (but cf. \cite{Iv} Section 3.5.4).
\item[-] For a stable almost Chebotarev set $S$, $E^{\rm sharp}(S)$ is finite and $E^{\rm strong}(S)$ is either $\Sigma_{\bQ}$ or finite (cf. Section \ref{sec:ex_propertiesast}).
\item[-] Roughly speaking, $p$-stability (for $\cL/K$) is enough to prove Hasse principles in dimension $1$ for $p$-primary ($\Gal_{\cL/K}$-)modules. Cf. Section \ref{sec:Shagroups_of_stable_sets}.
\item[-] To prove Hasse principles in dimension 2 and results of Grunwald-Wang type results for $p$-primary $\Gal_{K,S}$-modules, we need strong $p$-stability. We will give examples of persistent sets $S$ together with a finite set $T$ such that Grunwald-Wang (even stably) fails, i.e., $\coker^1(K_{S \cup T}/L, T; \bZ/p\bZ) \neq 0$ for all finite subextensions $K_S/L/K$. But it is not clear whether one can find such an example with additional requirement that $T \subseteq S$ (and necessarily $S$ being not strongly $p$-stable).  Cf. Section \ref{sec:GWT}.
\item[-] On the other side, for applications of Grunwald-Wang (i.e., to prove Riemann's existence theorem, to realize local extensions by $K_S/K$, to compute (strict) cohomological dimension, etc.), it is enough to require that $S$ is sharply $p$-stable. Cf. Sections \ref{sec:GWTRExT_overview}, \ref{sec:GWT_and_dagger} and \ref{sec:proof_of_dagger_results}.
\end{itemize}


\section{Examples} \label{sec:Examples_section}

In this section we construct examples of stable sets. First, in Section \ref{sec:Expls_Sets_of_delta_1} we see to which extent 'stable' is more general than 'of density 1'. Then, in Sections \ref{sec:ACS} and \ref{sec:ex_propertiesast} we introduce almost Chebotarev sets and determine when they are stable resp. strongly $p$-stable resp. sharply $p$-stable. Finally, in Section \ref{sec:stable_sets_with_NS=1} we construct a stable almost Chebotarev set $S$ with $\bN(S) = \{1\}$.

\subsection{Sets of density one}\label{sec:Expls_Sets_of_delta_1}

Stable and persistent sets generalize sets of density one. In particular, every set of primes of $K$ of density one is persistent for any extension $\cL/K$ with persisting field $K$ and is strongly $p$-stable for each $p$. Nevertheless, sets of density one have some properties, which stable resp. persistent sets do not have in general:

\begin{itemize}
 \item[(i)] The intersection of two sets of density one again has density one, which is not true for stable and persistent sets: the intersection of two sets persistent for $\cL/K$ can be empty (cf. Corollary \ref{cor:stability_of_ACS} and explicit examples below).
\item[(ii)] If $S \subseteq \Sigma_K$ has density one, then there are infinitely many primes $p \in \Sigma_{\bQ}$, such that $S_p \subseteq S$ (otherwise, for all primes $p \in \cs(K/\bQ)$ one could choose a prime $\fp \in S_p \sm S$ of $K$ and we would have $\delta_K(S) \leq 1 - [K^n:\bQ]^{-1}$, where $K^n$ denotes the normal closure of $K$ over $\bQ$). On the other side, it is easy to construct a persistent set $S \subseteq \Sigma_K$ with $\bN(S) = \{ 1 \}$, i.e., $S_{\ell} \not\subseteq S$ for all $\ell \in \Sigma_{\bQ}$ (cf. Section \ref{sec:stable_sets_with_NS=1} for an example).
\end{itemize}

Observe that for sets $S$ with $\bN(S) = \{ 1 \}$, mentioned above, no one of the $\ell$-adic representations $\rho_{A,\ell} \colon \Gal_K \rar \GL_d(\cool)$, which comes from an abelian variety $A/K$, factors through the quotient $\Gal_K \tar \Gal_{K,S}$ (indeed, the Tate-pairing on $A$ shows that the determinant of $\rho_{A,\ell}$ is the $\ell$-part of the cyclotomic character of $K$, and in particular, $\rho_{A,\ell}$ is highly ramified at all primes of $K$ lying over $\ell$. If $\rho_{A,\ell}$ would factor over $\Gal_{K,S}$, then we would have $S_{\ell} \subseteq S$). In particular, this makes it very hard, if not impossible, to study the group $\Gal_{K,S}$ via Langlands program (for example like in \cite{Ch} and \cite{CC}, where a prime $\ell \in \bN(S)$ is always necessary). If $S$ is additionally stable, then methods involving stability allow to study $\Gal_{K,S}$.


\subsection{Almost Chebotarev sets}\label{sec:ACS}

\begin{Def}
Let $K$ be a number field and $S$ a set of primes of $K$. Then $S$ is called a \textbf{Chebotarev set} (resp. an \textbf{almost Chebotarev set}), if $S = P_{M/K}(\sigma)$ (resp. $S \backsimeq P_{M/K}(\sigma)$), where $M/K$ is a finite Galois extension and $\sigma \in \Gal_{M/K}$.
\end{Def}

\begin{rem} $M$ and the conjugacy class of $\sigma$ are not unique, i.e., there are pairs $(M/K, \sigma)$, $(N/K, \tau)$ such that $M \neq N$ and $P_{M/K}(\sigma) \backsimeq P_{N/K}(\tau)$ (or even equal). If one restricts attention to pairs $(M/K,\sigma)$ such that $\sigma$ is central in $\Gal_{M/K}$, then $(M/K,\sigma)$ is indeed unique. Cf. \cite{Iv} Remark 3.13.
\end{rem}

\begin{prop}\label{prop:compdens}
Let $M/K$ be a finite Galois extension, $\sigma \in \Gal_{M/K}$ and $L/K$ any finite extension. Let $L_0 := L \cap M$. Then

\[
\delta_L(P_{M/K}(\sigma)_L) = \frac{ \abs{ C(\sigma; \Gal_{M/K}) \cap \Gal_{M/L_0} } }{ \abs{\Gal_{M/L_0}} }.
\]
\end{prop}

\noindent Thus $\delta_L(P_{M/K}(\sigma)_L) \neq 0$ if and only if $C(\sigma; \Gal_{M/K}) \cap \Gal_{M/L_0} \neq \emptyset$. In particular, this is always the case if $L_0 = K$ or if $\sigma = 1$.

\begin{proof}
Let $N/K$ be a finite Galois extension with $N \supseteq ML$. Let $H := \Gal_{N/L}$ and $\overline{H} := \Gal_{M/L_0}$. We have a natural surjection $H \tar \overline{H}$. Let $\bold{1}_{\sigma}$ denote the class function on $\Gal_{M/K}$, which has value $1$ on $C(\sigma; \Gal_{M/K})$ and 0 outside. Finally, let $m_H$ denote the character on $G := \Gal_{N/K}$ of the induced representation $\Ind_H^G \bold{1}_H$. Then we have (the first equality below follows from \cite{Wi} Proposition 2.1 and the second from Lemma \ref{lm:Dichtepullback}):
\begin{eqnarray*}
\delta_L(P_{M/K}(\sigma)_L) &=& \sum_{C(g; G) \mapsto C(\sigma; \Gal_{M/K})} \delta_L(P_{N/K}(g)_L) \\
&=& \sum_{C(g; G) \mapsto C(\sigma; \Gal_{M/K})} m_H(g) \delta_K(P_{N/K}(g)) \\
&=& \sum_{C(g; G) \mapsto C(\sigma; \Gal_{M/K})} m_H(g) \frac{ \abs{C(g;G)} }{ \abs{ G } } \\
&=& \frac{1}{ \abs{G} } \sum_{g \mapsto C(\sigma; \Gal_{M/K}) } m_H(g)  \\
&=& \langle m_H, \inf\nolimits_{\Gal_{M/K}}^G \bold{1}_{\sigma} \rangle_G \\
&=& \langle \bold{1}_H, \inf\nolimits_{\Gal_{M/K}}^H \bold{1}_{\sigma} \rangle_H \\
&=& \langle \bold{1}_{\overline{H}}, \bold{1}_{\sigma}|_{\overline{H}} \rangle_{\overline{H}} \\
&=& \frac{ \abs{ C(\sigma; \Gal_{M/K}) \cap \overline{H} } }{ \abs{\overline{H}} },
\end{eqnarray*}
\noindent where the third-to-last equality is Frobenius reciprocity, and the second-to-last equality follows from the easy fact that if $H \tar \overline{H}$ is a surjection of finite groups, $\chi, \rho$ are two characters of $\overline{H}$, then $\langle \inf\nolimits_{\overline{H}}^H \chi, \inf\nolimits_{\overline{H}}^H \rho \rangle_H = \langle \chi, \rho \rangle_{\overline{H}}$.
\end{proof}

\begin{cor}\label{cor:stability_of_ACS}
Let $M/K$ be finite Galois and let $\sigma \in \Gal_{M/K}$. Let $\cL/K$ be any extension and set $L_0 := M \cap \cL$. Then a set $S \backsimeq P_{M/K}(\sigma)$ is persistent for $\cL/K$ if and only if \[C(\sigma ; \Gal_{M/K}) \cap \Gal_{M/L_0} \neq \emptyset.\] If this is the case, $L_0$ is a persistent field for $S$ for $\cL/K$. In particular,
\begin{itemize}
\item[(i)]  any set $S \backsimeq \cs(M/K)$ is persistent for any extension $\cL/K$,
\item[(ii)] any set $S \backsimeq P_{M/K}(\sigma)$ is persistent for any extension $\cL/K$ with $\cL \cap M = K$.
\end{itemize}
\end{cor}

\begin{ex}(A persistent set) Let $K$ be a number field, $M/K$ a finite Galois extension, which is totally ramified in a prime $\fp$ of $K$. Let $\sigma \in \Gal_{M/K}$ and let $S$ be a set of primes of $K$, such that $S \backsimeq P_{M/K}(\sigma)$ and $\fp \not\in S$. Then $S$ is persistent with persisting field $K$. Indeed, we have $K_S \cap M = K$ by construction, and the claim follows from Corollary \ref{cor:stability_of_ACS}.
\end{ex}

\begin{cor}\label{prop:stabioity_properties_of_ACS}
Let $S$ be an almost Chebotarev set and $\cL/K$ an extension. Then the following are equivalent:
\begin{itemize}
\item[(i)]   $S$ is stable for $\cL/K$.
\item[(ii)]  $S$ is persistent for $\cL/K$.
\item[(iii)] $\delta_L(S) > 0$ for all finite $\cL/L/K$.
\end{itemize}

\end{cor}

\begin{proof}
Let $S \backsimeq P_{M/K}(\sigma)$ with a finite Galois extension $M/K$ and $\sigma \in \Gal_{M/K}$. By Proposition \ref{prop:compdens}, the density of $S$ is constant and equal to some $d \geq 0$ in the tower $\cL/L_0$ with $L_0 = \cL \cap M$. There are two cases: either $d = 0$ or $d > 0$. If $d = 0$, then $S$ is not stable and hence also not persistent for $\cL/K$ by Proposition \ref{prop:stableequalbounded}, i.e., (i), (ii) and (iii) do not hold in this case. If $d > 0$, then $S$ is obviously persistent for $\cL/K$ with persisting field $L_0$ and hence also stable, i.e., (i),(ii),(iii) hold. \qedhere

\end{proof}

\begin{rem} If $S$ is \emph{any} stable set, then (ii) $\Rar$ (i) $\Rar$ (iii) still holds. But (iii) $\Rar$ (i) fails in general (cf. \cite{Iv} Section 3.5.4) and it is not clear whether (i) $\Rar$ (ii) holds.
\end{rem}


\subsection{Finiteness of $E^{\rm sharp}(S)$ and $E^{\rm strong}(S)$. Examples} \label{sec:ex_propertiesast}

\begin{prop}\label{prop:finitnessofES} Let $S \backsimeq P_{M/K}(\sigma)$ with $\sigma \in \Gal_{M/K}$.
\begin{itemize}
\item[(i)] If $\infty \in E^{\rm strong}(S)$, then $E^{\rm strong}(S)$ contains all rational primes. If $\infty \not\in E^{\rm strong}(S)$, then $E^{\rm strong}(S)$ is finite.



\item[(ii)] Assume $S$ is stable. If $\mu_p \subset K_S$ or $M/K$ unramified in $S_p \sm S$, then $S$ is sharply $p$-stable. In particular, if $S$ is stable, then $E^{\rm sharp}(S)$ is finite.
\end{itemize}
\end{prop}

\begin{proof}
(i): If $\infty \in E^{\rm strong}(S)$, then $S$ does not have a stabilizing field for $K_{S \cup S_{\infty}}/K$, which is contained in $K_S$. This is by Proposition \ref{prop:stableequalbounded} equivalent to the fact that $S$ is not stable for $K_{S \cup S_{\infty}}/K$, which in turn is equivalent by Corollary \ref{prop:stabioity_properties_of_ACS} to the fact that $\delta_L(S) = 0$ for all $K_{S \cup S_{\infty}}/L/L_0$, where $L_0$ is some finite subextension of $K_{S \cup S_{\infty}}/K$. Thus $p \in E^{\rm strong}(S)$ for any $p$.

Now assume $\infty \not\in E^{\rm strong}(S)$. Let $L_0 := M \cap K_{S \cup S_{\infty}}$ and $L_p := M \cap K_{S \cup S_p \cup S_{\infty}}$. By Proposition \ref{prop:compdens}, the density of $S$ is constant in the towers $K_{S \cup S_{\infty}}/L_0$ and $K_{S \cup S_p \cup S_{\infty}}/L_p$ and equal to some real numbers $d_0$ and $d_p$ respectively. Since $S$ is stable for $K_{S \cup S_{\infty}}/K$, we have $d_0 > 0$.

We claim that for almost all $p$'s we have $L_p = L_0$. More precisely, this is true for all $p$'s, such that the set

\[ \{\fp \in (S_p \sm S)_{L_0} \colon \fp \text{ is ramified in } M/L_0 \}. \]

\noindent is empty. In fact, if this set is empty for $p$, then the extension $L_p/L_0$ is unramified in $S_p \sm S(L_0)$, since contained in $M/L_0$. But being contained in $K_{S \cup S_p \cup S_{\infty}}$ and unramified in $S_p \sm S(L_0)$, it is contained in $K_{S \cup S_{\infty}}$, and hence also in $M \cap K_{S \cup S_{\infty}} = L_0$, which proves our claim.

Let now $p$ be such that $L_p = L_0$. Then we claim that $S$ is $([L_0:K]d_0^{-1})$-stable for $K_{S \cup S_p \cup S_{\infty}}/K$ with $([L_0:K]d_0^{-1})$-stabilizing field $K$. Indeed, as $L_p = L_0$, we have $d_p = d_0 > 0$. Let $K_{S \cup S_p \cup S_{\infty}}/N/K$ be any finite subextension. We have

\[ d_0 = \delta_{L_0 N}(S) = [L_0 N : N]\delta_N(S \cap \cs(L_0 N/N)) \leq [L_0 : K] \delta_N(S), \]

\noindent i.e., $\delta_N(S) \geq [L_0 : K]^{-1} d_0$ for all $N$, and our claim follows.

Finally, almost all primes satisfy $p > [L_0 : K] d_0^{-1}$ and $L_p = L_0$. For such primes $S$ is $p$-stable for $K_{S \cup S_p \cup S_{\infty}}/K$ with stabilizing field $K$.

(ii): the second assertion of (ii) follows from the first. If $\mu_p \subseteq K_S$, then $S$ is sharply $p$-stable by Corollary \ref{prop:stabioity_properties_of_ACS}. Assume $M/K$ is unramified in $S_p \sm S$. Let $L_0 := M \cap K_S$, $L_0^{\prime} := L_0(\mu_p) \cap K_S$ and $L_p := M \cap K_S(\mu_p)$. From these definitions resp. from our assumption on $M/K$ we have: (1) $\Gal_{K_S(\mu_p)/L_0^{\prime}} \cong \Gal_{K_S/L_0^{\prime}} \times \Gal_{L_0(\mu_p)/L_0^{\prime}}$ and $L_0(\mu_p)/L_0^{\prime}$ has no subextension unramified in $S_p \sm S$, (2) $L_p \cap K_S = L_0$ and (3) $L_p/L_0$ is unramified in $S_p \sm S$. By (3) the extension $L_p.L_0^{\prime}/L_0^{\prime}$ is unramified in $S_p \sm S$ and by (1) we get $L_p \subseteq L_p.L_0^{\prime} \subseteq K_S$. Hence (2) gives $L_p = L_0$. Thus for all $K_S(\mu_p)/L/L_0$ we have by Proposition \ref{prop:compdens}: $\delta_L(S) = \delta_{L_p}(S) = \delta_{L_0}(S) > 0$ since $S$ is stable. \qedhere

\end{proof}

\begin{rem} Let $S \backsimeq P_{M/K}(\sigma)$. We have the following equivalences:
\[ p \not\in E^{\rm sharp}(S) \LRar \text{$S$ stable for $K_S(\mu_p)/K$ } \LRar C(\sigma ; \Gal_{M/K}) \cap \Gal(M/M \cap K_S(\mu_p))\neq \emptyset. \]

\end{rem}

\begin{ex}(Persistent sets with $E^{\rm strong}(S)$ finite but non-empty)\label{ex:pers_ES_fin_and_nonempty}
Let $K$ be a totally imaginary number field and let $M/K$ be a finite Galois extension, which satisfies the following conditions:
\begin{itemize}
\item $M/K$ is totally ramified in a prime $\fp \in S_p(K)$,
\item $d := [M:K] > p$.
\end{itemize}

\noindent Let $\sigma \in \Gal_{M/K}$ and let $S$ be a set of primes of $K$, such that
\begin{itemize}
\item $S \backsimeq P_{M/K}(\sigma)$,
\item $\Ram(M/K) \sm S = \{ \fp \}$.
\end{itemize}

\noindent Then $S$ is persistent ($\delta_L(S) = d^{-1}$ for all $K_S/L/K$) with persisting field $K$. Further, $S$ is not strongly $p$-stable, i.e., $p \in E^{\rm strong}(S)$ and $\infty \not\in E^{\rm strong}(S)$, i.e., $E^{\rm strong}(S)$ is finite by Proposition \ref{prop:finitnessofES}. Indeed, $M \subseteq K_{S \cup S_p \cup S_{\infty}}$ and there are two cases $\sigma = 1$ or $\sigma \neq 1$. In the second case, the density of $S$ in $K_{S \cup S_p \cup S_{\infty}}/K$ is zero beginning from $M$, hence $S$ is non-stable for this extension, and $S$ is not strongly $p$-stable. In the first case, we have $\delta_L(S) = 1$ for all $K_{S \cup S_p \cup S_{\infty}}/L/M$. Assume there is a $p$-stabilizing field $N \subseteq K_S$ for $S$ for $K_{S \cup S_p \cup S_{\infty}}/K$, i.e., there is some $S_0 \subseteq S$ and some $a \in (0,1]$ with $a \leq \delta_L(S_0) < pa$ for all $K_{S \cup S_p \cup S_{\infty}}/L/N$. But this leads to a contradiction. Indeed,

\[ \delta_{MN}(S_0) = [MN:N]\delta_N(S_0 \cap \cs(MN/N)) = [M:K]\delta_{N}(S_0) \geq p\delta_N(S_0),\]

\noindent since $N \cap M = K$ and $S_0 \subseteq S \backsimeq \cs(M/K)$.
\end{ex}

\begin{ex}(Persistent sets with $E^{\rm strong}(S) = \emptyset$)  \label{ex:perswithemptyE1} Let $M/K$ be a finite Galois extension of degree $d$ with $K$ totally imaginary, which is totally ramified in at least two primes $\fp$ resp. $\fl$ with different residue characteristics $\ell_1$ resp. $\ell_2$. Let $S \backsimeq P_{M/K}(\sigma)$ for some $\sigma \in \Gal_{M/K}$, such that $\fp, \fl \not\in S$. Then $M \cap K_S = K$, hence $S$ is persistent with persisting field $K$. Let $p$ be a rational prime. Then $M \cap K_{S \cup S_p \cup S_{\infty}} = K$, since $M/K$ is totally ramified over primes with \emph{different} residue characteristics $\ell_1$ and $\ell_2$. Hence $S$ is strongly $p$-stable for every prime $p$ and $K$ is a persisting field for $S$ for $K_{S \cup S_p \cup S_{\infty}}/K$.
\end{ex}

\begin{ex}(Persistent sets with $E^{\rm strong}(S) = \emptyset$) \label{ex:perswithemptyE2} There is also another possibility to construct sets $S$ with $E^{\rm strong}(S) = \emptyset$, using the same idea as in the preceding example. Assume for simplicity that $K$ is totally imaginary. Let $M_1, M_2/K$ be two Galois extensions of $K$, and $\sigma_1 \in \Gal_{M_1/K}, \sigma_2 \in \Gal_{M_2/K}$. Assume $M_i/K$ is totally ramified in a non-archimedean prime $\fp_i$ of $K$, such that the residue characteristics of $\fp_1, \fp_2$ are unequal. Then let $S$ be a set of primes of $K$, such that
\begin{itemize}
\item $S \supsetsim P_{M_1/K}(\sigma_1) \cup P_{M_2/K}(\sigma_2)$,
\item $\{ \fp_1, \fp_2 \} \not\in S$.
\end{itemize}

\noindent Then, by the same reasoning as in the preceding example, $S$ is persistent with persisting field $K$ and $E^{\rm strong}(S) = \emptyset$. Moreover for each rational prime $p$, the field $K$ is persisting for $S$ for $K_{S \cup S_p \cup S_{\infty}}/K$.
\end{ex}

\subsection{Stable sets with $\bN(S) = \{ 1 \}$}\label{sec:stable_sets_with_NS=1}

Let $M/K/K_0$ be two finite Galois extensions of a number field $K_0$. Then the natural map $\Gal_{M/K_0} \rar \Aut(\Gal_{M/K})$ induces an exterior action

\[ \Gal_{K/K_0} \rar \Out(\Gal_{M/K}), \]

\noindent thus inducing a natural action of $\Gal_{K/K_0}$ on the set of all conjugacy classes of $\Gal_{M/K}$. For any $g \in \Gal_{K/K_0}$ and $\sigma \in \Gal_{M/K}$, we choose a representative of the conjugacy class $g.C(\sigma; \Gal_{M/K})$ and denote it by $g.\sigma$. Further, $\Gal_{K/K_0}$ acts naturally on $\Sigma_K$, and we have

\[ g.P_{M/K}(\sigma) = P_{M/K}(g.\sigma). \]

\noindent Let $K_0 = \bQ$ and let $\sigma \in \Gal_{M/K}$ be an element, such that $C(\sigma; \Gal_{M/K})$ is not fixed by the action of $\Gal_{K/\bQ}$. Let then
\[ S := \cs(K/\bQ)_K \cap P_{M/K}(\sigma). \]

\noindent If $p \in \Sigma_{\bQ,f} \sm \cs(K/\bQ)$, then $S \cap S_p = \emptyset$. If $p \in \cs(K/\bQ)$ such that $S_p \cap S \neq \emptyset$, then the action of $g \in \Gal_{K/K_0}$, chosen such that $C(\sigma; \Gal_{M/K}) \neq C(g.\sigma; \Gal_{M/K})$, defines an isomorphism between the disjoint sets $S_p \cap P_{M/K}(\sigma)$ and $S_p \cap P_{M/K}(g.\sigma)$, hence the last of these two sets is non-empty. From this we obtain $S_p \not\subseteq S$. Thus $\bN(S) = \{ 1 \}$. Moreover, if we choose $\sigma$ such that the stabilizer of $C(\sigma; \Gal_{M/K})$ in $\Gal_{K/\bQ}$ is trivial, then for any $p$ the intersection $S_p \cap S$ is either empty or contains exactly one element.

Now we have to choose $M$ in a way such that $S$ is stable. This is easy: e.g, take $M/K$ such that it is totally ramified in a fixed prime, which is (by definition of $S$) not contained in $S$. Then $K_S \cap M =  K$, i.e., $S$ is stable for $K_S/K$ with stabilizing field $K$, as $\delta_K(\cs(K/\bQ)_K) = 1$ and hence $S \backsimeq P_{M/K}(\sigma)$.




\section{Shafarevich-Tate groups of stable sets}\label{sec:Shagroups_of_stable_sets}

In this section we generalize many Hasse principles to stable sets and additionally prove finiteness resp. uniform bounds of certain Shafarevich-Tate groups associated with stable sets. The main result is the Hasse principle in Theorem \ref{thm:keyresult}. Further, there are two variants of uniform bounds on the size of $\Sha^i$: on the one side one can vary the coefficients, and on the other side the base field. We study both variants, the first in Section \ref{sec:Finitness_of_Sha_div_coeff} and the second in Section \ref{sec:unif_bounds_on_Sha1_field_var}. These results are used in later sections.

\subsection{Stable sets and $\Sha^1$: key result} \label{sec:stablesets_and_Sha1_key_result}

Let $K$ be a number field and $\cL/K$ a (possibly infinite) Galois extension. Let $A$ be a finite $\Gal_{\cL/K}$-module. Let now $T$ be a set of primes of $K$. Consider the \textbf{$i$-th Shafarevich-Tate group} with respect to $T$:

\[ \Sha^i(\cL/K, T; A) := \ker(\res^i \colon \coh^i(\cL/K, A) \rar \prod_{\fp \in T} \coh^i(\cG_{\fp}, A) ), \]

\noindent where $\cG_{\fp} = \Gal_{ K_{\fp}^{\sep}/K_{\fp} }$ is the local absolute Galois group (the map $\res$ is essentially independent of the choice of this separable closure, and we suppress it in the notation). We also write $\Sha^i(K_S/K;A)$ instead of $\Sha^i(K_S/K, S; A)$. We denote by $K(A)$ the \textbf{trivializing extension} for $A$, i.e., the smallest field between $K$ and $\cL$, such that the subgroup $\Gal_{\cL/K(A)}$ of $\Gal_{\cL/K}$ acts trivially on $A$. It is a finite Galois extension of $K$.

Let $G$ be a finite group and $A$ a $G$-module. Following Serre \cite{Se}\S2 and Jannsen \cite{Ja}, let $\coh^i_{\ast}(G,A)$ be defined by exactness of the following sequence:
\[ 0 \rar \coh^i_{\ast}(G,A) \rar \coh^i(G,A) \rar \prod_{\stackrel{H \subseteq G}{\text{cyclic}}} \coh^i(H,A). \]

Our key result is the following theorem. All results in the following make use of this theorem in a crucial way.

\begin{thm}\label{thm:keyresult}
Let $K$ be a number field, $T$ a set of primes of $K$ and $\cL/K$ a Galois extension. Let $A$ be a finite $\Gal_{\cL/K}$-module. Assume that $T$ is $p$-stable for $\cL/K$, where $p$ is the smallest prime divisor of $\abs{A}$. Let $L$ be a $p$-stabilizing field for $T$ for $\cL/K$. Then:

\[ \Sha^1(\cL/L,T; A) \subseteq \coh^1_{\ast}(L(A)/L, A). \]

\noindent In particular, if $\coh^1_{\ast}(L(A)/L, A) = 0$, then $\Sha^1(\cL/L,T; A) = 0$.
\end{thm}

\begin{lm}\label{lm:shortShaseq}
Let $\cL/L/K$ be two Galois extensions of $K$ and $T$ a set of primes of $K$. Let $A$ be a $\Gal_{\cL/K}$-module, such that for any $\fp \in T$ one has $A^{\Gal_{\cL/L}} = A^{D_{\fp, \cL/L}}$. Then there is an exact sequence
\[ 0 \rar \Sha^1(L/K, T; A^{\Gal_{\cL/L}}) \rar \Sha^1(\cL/K, T; A) \rar \Sha^1(\cL/L, T_L; A) \]
\end{lm}

\begin{proof}
The proof is an easy and straightforward exercise.
\end{proof}

\begin{lm}\label{lm:Sha_in_coh_ast}
Let $L/K$ be a finite Galois extension, $T$ a set of primes of $K$, and $A$ a finite $\Gal_{L/K}$-module and $i > 0$. Assume that $T$ is $p$-stable for $L/K$ with $p$-stabilizing field $K$, where $p$ is the smallest prime divisor of $\abs{A}$. Then
\[ \Sha^i(L/K, T; A) \subseteq \coh^i_{\ast}(L/K, A). \]
\end{lm}

\begin{proof} Since any $p$-stable set is $\ell$-stable for all $\ell > p$, we can assume that $A$ is $p$-primary. We have to show that any cyclic $p$-subgroup of $\Gal_{L/K}$ is a decomposition subgroup of a prime in $T$. This is the content of the next lemma.
\end{proof}

\begin{lm}
Let $L/K$ be a finite Galois extension, $T$ a set of primes of $K$ and $p$ a rational prime, such that $T$ is $p$-stable for $L/K$ with $p$-stabilizing field $K$. Then any cyclic $p$-subgroup of $\Gal_{L/K}$ is the decomposition group of a prime in $T$.
\end{lm}

\begin{rem}
\begin{itemize}
\item[(i)] This shows automatically that there are infinitely many primes in $T$, for which the given cyclic group is a decomposition group.
\item[(ii)] In some sense this lemma 'generalizes' Chebotarev's density theorem, which says in particular, that if $S$ has density one and $L/K$ is finite Galois, then any element of $\Gal_{L/K}$ is a Frobenius of a prime in $S$.
\end{itemize}
\end{rem}

\begin{proof}
Assume that the cyclic $p$-subgroup $H \subseteq \Gal_{L/K}$ is not a decomposition group of a prime in $T$. Let $pH \subseteq H$ be the subgroup of index $p$. Then one computes directly $m_{pH}(\sigma) = p m_H(\sigma)$ for any $\sigma \in pH$. Since $H$ is not a decomposition subgroup of a prime $\fp \in T$, no generator of $H$ is a Frobenius at $T$, i.e., $P_{L/K}(\sigma) \cap T = \emptyset$ for any $\sigma \in H \sm pH$. By $p$-stability of $T$, there is a subset $T_0 \subseteq T$ and an $a > 0$, such that $pa > \delta_{L^{\prime}}(T_0) \geq a$ for all $L/L^{\prime}/K$.  Let $L_0 = L^H$ and $L_1 = L^{pH}$. Then by Lemma \ref{lm:Dichtepullback}

\begin{eqnarray*}
\delta_{L_0}(T_0) &=& \sum_{\sigma \in H} m_H(\sigma)\delta_K(P_{L/K}(\sigma) \cap T_0) \\
&=& \sum_{\sigma \in pH} m_H(\sigma)\delta_K(P_{L/K}(\sigma) \cap T_0) \\
&=& p^{-1} \sum_{\sigma \in pH} m_{pH}(\sigma)\delta_K(P_{L/K}(\sigma) \cap T_0) \\
&=& p^{-1} \delta_{L_1}(T_0).
\end{eqnarray*}

\noindent This contradicts our assumption on $T_0$.

\end{proof}

\begin{proof}[Proof of Theorem \ref{thm:keyresult}]
We can assume $L = K$.
By applying Lemma \ref{lm:shortShaseq} to $\cL/K(A)/K$ and using Lemma \ref{lm:Sha_in_coh_ast}, we are reduced to showing that if $A$ is a trivial $G$-module, then $\Sha^1(\cL/K,T;A) = 0$. Let $T_0 \subseteq T$ and $a > 0$ be such that $pa > \delta_{L^{\prime}}(T_0) \geq a$ for all $\cL/L^{\prime}/K$. Let $\Gal_{\cL/K}^T$ be the quotient of $\Gal_{\cL/K}$, corresponding to the maximal subextension of $\cL/K$, which is completely split in $T$. We have then

\[\Sha^1(\cL/K,T;A) = \ker( \Hom(\Gal_{\cL/K}, A) \rar \prod_{\fp \in T} \Hom(\cG_{\fp}, A)) = \Hom(\Gal_{\cL/K}^T, A).\]

\noindent If $0 \neq \phi \in \Hom(\Gal_{\cL/K}^T, A)$, then $M := \cL^{\ker(\phi)}/K$ is a finite extension inside $\cL/K$ with Galois group $\im(\phi) \neq 0$ and completely decomposed in $T$, and in particular in $T_0$. Thus

\[ pa > \delta_M(T_0) = [M:K]\delta_K(T_0 \cap \cs(M/K)) = \abs{\im(\phi)}\delta_K(T_0) \geq pa, \]

\noindent since $\delta_K(T_0) \geq a$. This is a contradiction, and hence we obtain
\[ \Sha^1(\cL/K,T;A) = \Hom(\Gal_{\cL/K}^T, A) = 0. \qedhere \]

\end{proof}


\subsection{Hasse principles} \label{sec:concrete_HPs_for_stablesets}

Let $K,S,T$ be a number field and two sets of primes of $K$. Various conditions on $S,T,A$ which imply the Hasse principles in cohomological dimensions 1 and 2 are considered in \cite{NSW} Chapter IX, \S1. We prove analogous results for stable sets. Before stating them, we refer the reader to \cite{NSW} 9.1.5, 9.1.7 for the definitions of the special cases.

\begin{cor}\label{cor:limHasseforstables} 
Let $K$ be a number field, $T, S$ sets of primes of $K$, $A$ a finite $\Gal_{K,S}$-module. Assume that $T$ is $p$-stable for $K_S/K$, where $p$ is the smallest prime divisor of $\abs{A}$. If $L$ is a $p$-stabilizing field for $T$ for $K_S/K$ and $\coh^1_{\ast}(L(A)/L, A) = 0$, then
\[ \Sha^1(K_S/L,T;A) = 0. \]
In particular, the following holds.
\begin{itemize}
\item[(i)]  Let $L_0$ be a $p$-stabilizing field for $T$ for $K_S/K$, which trivializes $A$. Then $\Sha^1(K_S/L,T;A) = 0$ for any finite $K_S/L/L_0$.
\item[(ii)] Assume $S \supseteq S_{\infty}$ and $n \in \bN(S)$ with the smallest prime divisor equal $p$. If $L_0$ is a $p$-stabilizing field for $T$ for $K_S/K$, then $\Sha^1(K_S/L, T; \mu_n) = 0$ for any finite $K_S/L/L_0$, such that we are not in the special case $(L,n,T)$. In the special case $(L,n,T)$ we have $\Sha^1(K_S/L, T; \mu_n) = \bZ/2\bZ$.
\end{itemize}
\end{cor}

The same also holds, if one replaces $\Gal_{K,S}$ by the quotient $\Gal_{K,S}(\fc)$, where $\fc$ is a full class of finite groups in the sense of \cite{NSW} 3.5.2.

\begin{proof} The first statement follows directly from Theorem \ref{thm:keyresult}. (i): since $L_0$ is a $p$-stabilizing field trivializing $A$, any finite subextension $L$ of $K_S/L_0$ has the same property. Hence (i) follows. To prove (ii), we can assume $n = p^r$. If we are not in the special case $(L, p^r)$, Proposition \cite{NSW} 9.1.6 implies $\coh^1(L(\mu_{p^r})/L, \mu_{p^r}) = 0$, i.e., we are done by Theorem \ref{thm:keyresult}. Assume we are in the special case $(L, p^r)$. In particular, $p = 2$. Then $\coh^1(L(\mu_{2^r})/L, \mu_{2^r}) = \bZ/2\bZ$. Since
\[ \Sha^1(K_S/L(\mu_{2^r}), T; \mu_{2^r}) = 0 \]

\noindent by Theorem \ref{thm:keyresult}, we see from Lemma \ref{lm:shortShaseq}
\[ \Sha^1(K_S/L, T; \mu_{2^r}) = \Sha^1(L(\mu_{2^r})/L, T; \mu_{2^r}). \]

\noindent Now the same argument as in the proof of \cite{NSW} 9.1.9(ii) finishes the proof. \qedhere

\end{proof}

Now we turn to $\Sha^2$. For a $\Gal_{K,S}$-module $A$, such that $\abs{A} \in \bN(S)$, we denote by
\[A^{\prime} := \Hom(A, \caO_{K_S,S}^{\ast})\]
\noindent the dual of $A$. As in \cite{NSW} 9.1.10, we obtain the following corollary.

\begin{cor}\label{cor:Sha2is0forstable}
Let $K$ be a number field, $S \supseteq S_{\infty}$ a set of primes of $K$, $A$ a finite $\Gal_{K,S}$-module with $\abs{A} \in \bN(S)$. Assume that $S$ is $p$-stable (i.e., $p$-stable for $K_S/K$), where $p$ is the smallest prime divisor of $\abs{A}$. Let $L$ be a $p$-stabilizing field for $S$ for $K_S/K$, such that
$\coh^1_{\ast}(L(A^{\prime})/L, A^{\prime}) = 0$. Then

\[ \Sha^2(K_S/L; A) = 0. \]
\noindent In particular:

\begin{itemize}
\item[(i)]  Let $L_0$ be a $p$-stabilizing field for $S$ for $K_S/K$, which trivializes $A^{\prime}$. Then $\Sha^2(K_S/L;A) = 0$ for any finite $K_S/L/L_0$.
\item[(ii)] Let $n \in \bN(S)$ with smallest prime divisor $p$. If $L$ is a $p$-stabilizing field for $S$ and we are not in the special case $(L, n, S)$, then $\Sha^2(K_S/L, \bZ/n\bZ) = 0$. In the special case, we have $\Sha^2(K_S/L; \bZ/n\bZ) = \bZ/2\bZ$.
\end{itemize}
\end{cor}

\begin{rem}\label{rem:bloedremark}
The condition $\abs{A} \in \bN(S)$ is not necessary if $A$ is trivial: we postpone the proof of this until all necessary ingredients (in particular Grunwald-Wang theorem,  Riemann's existence theorem and $\cd_p \Gal_{K,S} = 2$) are proven. Cf. Proposition \ref{prop:Sha2vanishingwoinverting}.
\end{rem}

\begin{proof}[Proof of Corollary \ref{cor:Sha2is0forstable}]
By Poitou-Tate duality \cite{NSW} 8.6.7 (this is the reason, why we need $S \supseteq S_{\infty}$ and $\abs{A} \in \bN(S)$) we have:

\[ \Sha^2(K_S/L, A) \cong \Sha^1(K_S/L,A^{\prime})^{\vee}, \]
where $X^{\vee} := \Hom(X, \bR/\bZ)$ is the Pontrjagin dual. An application of Theorem \ref{thm:keyresult} to $K_S/K$, the sets $S = T$ and the module $A^{\prime}$ gives the desired result. (i) and (ii) follow from Corollary \ref{cor:limHasseforstables}.
\end{proof}


\subsection{Finiteness of the Shafarevich-Tate group with divisible coefficients} \label{sec:Finitness_of_Sha_div_coeff}

As a version of Corollary \ref{cor:limHasseforstables}(i), we have the following proposition.

\begin{prop}\label{prop:boundofShafbystab}
Let $K$ be a number field, $\cL/K$ a Galois extension, $p^m$ some rational prime power ($m \geq 1$). Let $T$ be a set of primes of $K$, which is $p^m$-stable for $\cL/K$, with $p^m$-stabilizing field $L_0$. Then
\[ \abs{\Sha^1(\cL/L, T; \bZ/p^r\bZ)}  < p^m \]

\noindent for  any $r > 0$ and any finite subextension $\cL/L/L_0$.
\end{prop}

\begin{proof} Let $T_0 \subseteq T$ and $a > 0$ be such that $a \leq \delta_L(T_0) < p^m a$ for all finite $\cL/L/L_0$. Let $\cL/L/L_0$ be a finite extension. Assume that $\abs{\Sha^1(\cL/L,T; \bZ/p^r\bZ)} \geq p^m$. Then also
\[ \abs{\Sha^1(\cL/L, T_0; \bZ/p^r\bZ)} \geq p^m \]

\noindent and we have:

\[ \Sha^1(\cL/L, T_0; \bZ/p^r\bZ) \cong \Hom(\Gal_{\cL/L}^{T_0}(p), \bZ/p^r\bZ) = (\Gal_{\cL/L}^{T_0}(p)^{\ab}/p^r)^{\vee}. \]

\noindent Thus $\abs{\Sha^1(\cL/L, T_0; \bZ/p^r\bZ)}  \geq p^m$ implies $\abs{\Gal_{\cL/L}^{T_0}(p)^{\ab}/p^r} \geq p^m$, and if $M/L$ is the subextension of $\cL/L$, corresponding to $\Gal_{\cL/L}^{T_0}(p)^{\ab}/p^r$, then it has a finite subextension $M_1$ of degree $\geq p^m$, which is completely split in $T_0$, hence $\delta_{M_1}(T_0) \geq p^m \delta_L(T_0)$, which is a contradiction to $p^m$-stability of $T_0$.
\end{proof}

\begin{cor}
Let $K$ be a number field, $\cL/K$ a Galois extension, and $T$ a set of primes of $K$ stable for $\cL/K$. Then $\Sha^1(\cL/K, T; \bQ_p/\bZ_p)$ is finite for any $p$. Moreover, $\Sha^1(\cL/K, T; \bQ/\bZ)$ is finite.
\end{cor}

\begin{proof} For the first statement it is enough to show that $\abs{\Sha^1(\cL/K, T; \bZ/p^r\bZ)}$ is uniformly bounded for $r > 0$. By Proposition \ref{prop:stableequalbounded}, there is some $m \geq  1$, such that $K$ is a $p^m$-stabilizing field for $T$ for $\cL/K$. Then Proposition \ref{prop:boundofShafbystab} implies $\abs{\Sha^1(\cL/K, T; \bZ/p^r\bZ)} < p^m$. For the last statement, we decompose: $\Sha^1(\cL/K, T; \bQ/\bZ) = \bigoplus_p \Sha^1(\cL/K, T; \bQ_p/\bZ_p)$. The proven part shows that each of the summands is finite. Moreover, almost all are zero: there is some $\lambda > 1$, such that $K$ is $\lambda$-stabilizing field for $T$ for $\cL/K$. Thus for any $p \geq \lambda$, the group $\Sha^1(\cL/K, T; \bQ_p/\bZ_p)$ vanishes.
\end{proof}


\subsection{Uniform bound}\label{sec:unif_bounds_on_Sha1_field_var}

For later needs (cf. Section \ref{sec:GWT_and_dagger}) we prove the following uniform bounds. The results of this section were not part of \cite{Iv}.

\begin{prop} Let $\cM/\cL/K$ be Galois extensions, $A$ a finite $\Gal_{\cM/K}$-module and let $S$ be stable for $\cL(A)/K$. Then there is some $C > 0$ such that
\[ \abs{\Sha^1(\cM/L, S; A)} < C \]
\noindent for all finite subextensions $\cL/L/K$.
\end{prop}

\begin{proof}
For each $\cL/L/K$, Lemma \ref{lm:shortShaseq} applied to $\cM/L(A)/L$ gives an exact sequence
\begin{equation}\label{eq:shaseq_ochujet} 0 \rar \Sha^1(L(A)/L, S; A) \rar \Sha^1(\cM/L, S; A) \rar \Sha^1(\cM/L(A), S_{L(A)}; A). \end{equation}
Now $\Sha^1(L(A)/L, S; A) \subseteq \coh^1(L(A)/L, A)$ and $\Gal_{L(A)/L}$ is a subgroup of the finite group $\Gal_{K(A)/K}$, thus for all $\cL/L/K$, we have
\[ \abs{\Sha^1(L(A)/L, S; A)} < m := 1 + \max_{H \subseteq \Gal_{K(A)/K}} \coh^1(H,A). \]

As $S$ is stable for $\cL(A)/K$, by Proposition \ref{prop:stableequalbounded} there is some $\epsilon > 0$, such that $\delta_N(S) > \epsilon$ for all $\cL(A)/N/K(A)$. Suppose that $\abs{\Sha^1(\cM/L(A), S, A)}  \geq \epsilon^{-1}$ for some $\cL/L/K$. Then, exactly as in the proof of Proposition \ref{prop:boundofShafbystab}, there is an extension $M/L(A)$ of degree $\geq \epsilon^{-1}$, which is completely split in $S$. We obtain:

\[ \delta_M(S) = [M:L(A)]\delta_{L(A)}(S) > \epsilon^{-1}\epsilon = 1, \]

\noindent which is a contradiction. Taking into account equation \eqref{eq:shaseq_ochujet}, we obtain the statement of the proposition with respect to $C := m\epsilon^{-1}$.
\end{proof}

\begin{cor}\label{prop:uniformbound_forSha1_constcoeff}
Let $K$ be a number field, $S$, $T$ sets of primes of $K$ and $n$ a natural number.
\begin{itemize}
\item[(i)]  Assume that $K_S/\cL/K$ is a subextension such that $S$ is stable for $\cL/K$ and that $T$ has density $0$. Then there is some real $C > 0$, such that for any $\cL/L/K$ one has:
\[ \abs{\Sha^1(K_{S \cup T}/L, S \sm T, \bZ/n\bZ)} < C. \]
\item[(ii)] Assume that $T \supseteq (S_{\infty} \sm S)$ has density $0$ and $n \in \caO_{K,S \cup T}^{\ast}$. Let $K_S/\cL/K$ be a subextension such that $S$ is stable for $\cL(\mu_n)/K$. There is some real $C > 0$ such that for any $\cL/L/K$ one has:
\[ \abs{\Sha^1(K_{S \cup T}/L, S \sm T, \mu_n)} < C. \]
\end{itemize}
\end{cor}

\begin{rem}
The case $S$ stable for $\cL/K$, but not stable for $\cL(\mu_p)/K$ still remains mysterious: one neither can show such an uniform bound by the same methods, nor find counterexamples. Moreover, the same kind of arguments not even shows that $\Sha^1(K_{S \cup T}/K, S \sm T, \mu_p)$ must be finite.
\end{rem}




\section{Arithmetic applications} \label{sec:Arith_appl}

\subsection{Overview and results}\label{sec:GWTRExT_overview}

In this section we will be interested in the applications of the Hasse principles proven in the preceding section for stable sets. In particular, we will show two versions of the Grunwald-Wang theorem for them, with varying assumptions: we will have a strong Grunwald-Wang result if we assume strong $p$-stability (Section \ref{sec:GWT}) and only a weaker $\dirlim$-version (which is still enough for applications) after weakening the assumption to sharp $p$-stability (Section \ref{sec:GWT_and_dagger}). After this we will consider realization of local extensions, Riemann's existence theorem and the cohomological dimension of $\Gal_{K,S}$. For each of these three results there is a pro-finite and a pro-p version respectively. We state them below and give proofs in Section \ref{sec:proof_of_dagger_results}. Further, in Section \ref{sec:hasse_principle_for_Sha2_without_p_inv} we prove a Hasse principle for $\Sha^2$ for constant $p$-primary coefficients without the assumption $p \in \caO_{K,S}^{\ast}$ (cf. Corollary 
\ref{cor:Sha2is0forstable} and Remark \ref{rem:bloedremark}).

\begin{thm}\label{thm:LocExt_RExt_CD}\label{prop:locext_for_dagger_p_rel} \label{thm:REXT_for_dagger_p_rel} \label{thm:CD_for_dagger_p_rel}
Let $K$ be a number field, $p$ a rational prime and $T \supseteq S \supseteq R$ sets of primes of $K$ with $R$ finite.

\begin{itemize}
\item[($A_p$)] Assume $S$ is sharply $p$-stable for $K_S^R(p)/K$. Then
\[ K_S^R(p)_{\fp} = \begin{cases} K_{\fp}(p) & \text{if } \fp \in S \sm R \\ K_{\fp}^{\nr}(p) & \text{if } \fp \not\in S. \end{cases} \]
\item[(A)] Assume $S$ is sharply $p$-stable for $K_S^R/K$. Then
\[ K_{S,\fp}^R \supseteq \begin{cases} K_{\fp}(p) & \text{if } \fp \in S \sm R \\ K_{\fp}^{\nr}(p) & \text{if } \fp \not\in S. \end{cases} \]
\item[($B_p$)] Assume $S$ is sharply $p$-stable for $K_S^R(p)/K$. Then the natural map
\[ \phi_{T,S}^R(p) \colon \bigast\limits_{\fp \in R(K_S^R(p))} \cG_{\fp}(p) \ast \bigast\limits_{\fp \in (T \sm S)(K_S^R(p))} I_{\fp}(p) \stackrel{\sim}{\longrar} \Gal_{K_T(p)/K_S^R(p)} \]
\noindent is an isomorphism, where $I_{\fp}(p) := \Gal_{K_{\fp}(p)/K_{\fp}^{\nr}(p)} \subseteq \cG_{\fp}(p) := \Gal_{K_{\fp}(p)/K_{\fp}}$.
\end{itemize}
Let $K_T^{\prime}(p)/K_S^R$ denote the maximal pro-$p$ subextension of $K_T/K_S^R$.
\begin{itemize}
\item[(B)] Assume $S$ is sharply $p$-stable for $K_S^R/K$. Then the natural map
\[ \phi_{T,S}^R \colon \bigast\limits_{\fp \in R(K_S^R)} \cG_{\fp}(p) \ast \bigast\limits_{\fp \in (T \sm S)(K_S^R)} I_{\fp}^{\prime}(p) \stackrel{\sim}{\longrar} \Gal_{K_T^{\prime}(p)/K_S^R} \]
\noindent is an isomorphism, where $I_{\fp}^{\prime}(p)$ denotes the Galois group of the maximal pro-$p$ extension of $K_{S,\fp}^R$.
\end{itemize}
Assume $p$ is odd, or $K$ is totally imaginary.
\begin{itemize}
\item[($C_p$)] Assume $S$ is sharply $p$-stable for $K_S^R(p)/K$. Then
\[ \cd \Gal_{K,S}^R(p) = \scd \Gal_{K,S}^R(p) = 2. \]
\item[(C)] Assume $S$ is sharply $p$-stable for $K_S^R/K$. Then
\[ \cd_p \Gal_{K,S}^R = \scd_p \Gal_{K,S}^R = 2. \]
\end{itemize}
\end{thm}


\subsection{Grunwald-Wang theorem and strong $p$-stability} \label{sec:GWT}

Consider the cokernel of the global-to-local restriction homomorphism

\[ \coker^i(K_S/K, T; A) := \coker(\res^i \colon \coh^i(K_S/K, A) \rar \prod\nolimits_{\fp \in T}^{\prime} \coh^i(\cG_{\fp}, A) ), \]

\noindent where $A$ is a finite $\Gal_{K,S}$-module, $T \subseteq S$ and $\prod^{\prime}$ means that almost all classes are unramified. If $A$ is a trivial $\Gal_{K,S}$-module, then the vanishing of this cokernel is equivalent to the existence of global extensions unramified outside $S$, which realize given local extensions at primes in $T$. If $S$ has density $1$, the set $T$ is finite, $A$ is constant and we are not in a special case, this vanishing is essentially the statement of the Grunwald-Wang theorem. Certain conditions on $S,T,A$, under which this cokernel vanishes are considered in \cite{NSW} chapter IX \S2. All of them require $S$ to have certain minimal density. We prove analogous results for stable sets.

\begin{cor}\label{surj_of_H1_map_for_stable_sets}
Let $K$ be a number field, $T \subseteq S$ sets of primes of $K$ with $S_{\infty} \subseteq S$. Let $A$ be a finite $\Gal_{K,S}$-module with $\abs{A} \in \bN(S)$. Assume that $T$ is finite and $S$ is $p$-stable, where $p$ is the smallest prime divisor of $\abs{A}$. For any $p$-stabilizing field $L$ for $S$ for $K_S/K$, such that
$\coh^1_{\ast}(L(A^{\prime})/L,A^{\prime}) = 0$, we have:
\[\coker^1(K_S/L, T; A) = 0. \]
\end{cor}

\begin{proof} Since $T$ is finite and $S$ is $p$-stable for $K_S/K$, $S \sm T$ also is $p$-stable for $K_S/K$, and the $p$-stabilizing fields for $S$ and $S \sm T$ are equal. Let $L$ be as in the corollary. By Theorem \ref{thm:keyresult} applied to $K_S/L$, $S \sm T$ and $A^{\prime}$, we obtain  $\Sha^1(K_S/L, S \sm T; A^{\prime}) = 0$. Then \cite{NSW} 9.2.2 implies $\coker^1(K_S/L, T; A) = 0$.
\end{proof}

Now we give a generalization of \cite{NSW} 9.2.7.

\begin{thm} \label{thm:gruwi_cokerform}
Let $K$ be a number field, $S$ a set of primes of $K$. Let $T_0, T \subseteq S$ be two disjoint subsets, such that $T_0$ is finite. Let $p$ be a rational prime and $r > 0$ an integer. Assume $S \sm T$ is $p$-stable for $K_{S \cup S_p \cup S_{\infty}}/K$ with $p$-stabilizing field $L_0$, which is contained in $K_S$. Then for any finite $K_S/L/L_0$, such that we are not in the special case $(L, p^r, S \sm (T_0 \cup T))$, the canonical map

\[ \coh^1(K_S/L, \bZ/p^r\bZ ) \rar \bigoplus_{\fp \in T_0(L)} \coh^1(\cG_{\fp}, \bZ/p^r\bZ) \oplus \bigoplus_{\fp \in T(L)} \coh^1(\cI_{\fp}, \bZ/p^r\bZ)^{\cG_{\fp}} \]
\noindent is surjective, where $\cI_{\fp} \subseteq \cG_{\fp} = \Gal_{K_{\fp}^{\sep}/L_{\fp}}$ is the inertia subgroup. If we are in the special case $(L, p^r, S \sm (T_0 \cup T))$, then $p = 2$ and the cokernel of this map is of order 1 or 2.
\end{thm}

\begin{proof}
This follows from Corollary \ref{cor:limHasseforstables}(ii) in exactly the same way as \cite{NSW} 9.2.7 follows from \cite{NSW} 9.2.3(ii).
\end{proof}

\begin{rems} \mbox{}
\begin{itemize}
\item[(i)] If $\delta_K(T) = 0$, the condition ``$S \sm T$ is $p$-stable for $K_{S \cup S_p \cup S_{\infty}}/K$ with a $p$-stabilizing field contained in $K_S$''  is equivalent to ``$S$ is strongly $p$-stable''.
\item[(ii)]  If $\delta_K(S) = 1$ and $\delta_K(T) = 0$, then $L_0 = K$ is a persisting field for $S \sm T$ for any $\cL/K$ and the condition in the theorem is automatically satisfied. Thus our result is a generalization of \cite{NSW} 9.2.7. To show that it is a proper generalization, we give the following example. Let $N/M/K$ be finite Galois extensions of $K$, such that $N/K$ (and hence also $M/K$) is totally ramified in a non-archimedean prime $\fl$ of $K$, lying over the rational prime $\ell$. Let $\sigma \in \Gal_{M/K}$ and let $\tilde{\sigma} \in \Gal_{N/K}$ be a preimage of $\sigma$. Let $S \supseteq T$ be such that
\begin{itemize}
\item[-] $S \backsimeq P_{M/K}(\sigma)$, $\fl \not\in S$ and $T \backsimeq P_{M/K}(\sigma) \sm P_{N/K}(\tilde{\sigma})$.
\end{itemize}

\noindent  Then $S \sm T \backsimeq P_{N/K}(\tilde{\sigma})$ is persistent for $K_{S \cup S_p \cup S_{\infty}}/K$ for any $p \neq \ell$, and, moreover, $K$ is a persisting field (indeed, this follows from $K_{S \cup S_p \cup S_{\infty}} \cap N = K$). Hence the sets $S \supseteq T$ satisfies the conditions of the theorem with respect to each $p \neq \ell$. Observe that in this example $T$ is itself persistent for $K_{S \cup S_p \cup S_{\infty}}/K$ with persisting field $K$. In \cite{NSW} 9.2.7, the set $T$ must have density zero.


\end{itemize}
\end{rems}

From this we obtain the following classical form of the Grunwald-Wang theorem. The proof is the same as in \cite{NSW} 9.2.8.

\begin{cor}
Let $T \subseteq S$ be sets of primes of a number field $K$.  Let $A$ be a finite abelian group. Assume that $T$ is finite and that for any prime divisor $p$ of $\abs{A}$, $S$ is $p$-stable for $K_{S \cup S_p \cup S_{\infty}}/K$ with stabilizing field $K$. For all $\fp \in T$, let $L_{\fp}/K_{\fp}$ be a finite abelian extension, such that its Galois group can be embedded into $A$. Assume that we are not in the special case $(K, {\rm exp}(A), S \sm T)$. Then there exists a global abelian extension $L/K$ with Galois group $A$, unramified outside $S$, such that $L$ has completion $L_{\fp}$ at $\fp \in T$.
\end{cor}


\begin{ex}(A set with persistent subset for which Grunwald-Wang stably fails)
Let $p$ be an odd prime and assume $\mu_p \subset K$ (in particular, $K$ is totally imaginary and we can ignore the infinite primes). Let $S$ be a set of primes of $K$. Let $V = S_p \sm S$ and let $T \supseteq V$ be a finite set of primes of $K$. By \cite{NSW} 9.2.2 we have for all $K_S/L/K$ a short exact sequence (recall that $\mu_p \cong \bZ/p\bZ$ by assumption):

\[ 0 \rar \Sha^1(K_{S \cup T}/L, S \cup T; \bZ/p\bZ) \rar \Sha^1(K_{S \cup T}/L, S \sm T; \bZ/p\bZ) \rar \coker^1(K_{S \cup T}/L, T; \bZ/p\bZ)^{\vee} \rar 0. \]

\noindent Assume now that $S$ is $p$-stable with $p$-stabilizing field $K$. Then
\[\Sha^1(K_{S \cup T}/L, S \cup T; \bZ/p\bZ) \subseteq \Sha^1(K_S/L, S; \bZ/p\bZ) = 0\]

\noindent and hence we have

\[ \coker^1(K_{S \cup T}/L, T; \bZ/p\bZ) \cong \Sha^1(K_{S \cup T}/L,S \sm T; \bZ/p\bZ)^{\vee}. \]

\noindent We can find such a set $S$ for which one has additionally $\Sha^1(K_{S \cup T}/L,S \sm T; \bZ/p\bZ) \neq 0$ for each $K_S/L/K$. For an explicit example, let $K = \bQ(\mu_p)$ and let $T \supseteq S_p(K)$ ($S_p(K)$ consists of exactly one prime) be a finite set of primes of $K$. Let $M/K$ be a Galois extension of degree $p$ with $\emptyset \neq \Ram(M/K) \subseteq T$ (e.g., $M = \bQ(\mu_{p^2})$). Let $S := \cs(M/K)$. Then $M \cap K_S = K$ and hence $ML \cap K_S = L$ for each $K_S/L/K$. Thus $S$ is persistent with persisting field $K$. Further, $ML/L$ is a Galois extension of degree $p$, which is completely split in $S \sm T$ and unramified outside $S \cup T$, hence the subgroup of $\Gal_{K_{S \cup T}/ML} \subsetneq \Gal_{K_{S \cup T}/L}$ is the kernel of a nontrivial homomorphism $0 \neq \phi_M \in \Sha^1(K_{S \cup T}/L,S \sm T; \bZ/p\bZ)$. Hence this group is non-trivial.

Thus we have: $S$ is persistent but not strongly $p$-stable, in particular, no $p$-stabilizing field for $S \backsimeq S \cup T$ for $K_{S \cup S_p \cup S_{\infty}}/K$ is contained in $K_S$ and Grunwald-Wang does not hold for $S \cup T \supseteq T$ (i.e., the cokernel in Theorem \ref{thm:gruwi_cokerform} is non-zero). It is still unclear, whether there is an example of sets $\tilde{S} \supseteq \tilde{T}$ such that $\tilde{S}$ is persistent but not strongly $p$-stable and Grunwald-Wang fails for $\tilde{S} \supseteq \tilde{T}$.
\end{ex}

Finally, we have two corollaries generalizing \cite{NSW} 9.2.4 and 9.2.9 to stable sets.

\begin{cor}\label{cor:GWT_for_induced_module}
Let $K$ be a number field, $T \subseteq S$ sets of primes of $K$ with $T$ finite. Let $K_S/L/K$ be a finite Galois subextension with Galois group $G$. Let $p$ be a prime and $A = \bF_p[G]^n$ a $\Gal_{K,S}$-module. Assume $S$ is $p$-stable for $K_{S \cup S_p \cup S_{\infty}}/K$ with $p$-stabilizing field $L$. Then the restriction map

\[ \coh^1(K_S/K, A) \rar \bigoplus_{\fp \in T} \coh^1(\cG_{\fp}, A) \]
\noindent is surjective.
\end{cor}

\begin{proof}[Proof (cf. \cite{NSW} 9.2.4).]
We have the commutative diagram, in which the vertical maps are Shapiro-isomorphisms:

\centerline{
\begin{xy}
\xymatrix{
\coh^1(K_S/K, A) \ar[d]^{\sim} \ar[r] & \bigoplus\limits_{\fp \in T} \coh^1(\cG_{\fp}, A) \ar[d]^{\sim} \\
\coh^1(K_S/L, \bF_p^n) \ar[r] & \bigoplus\limits_{\fP \in T(L)} \coh^1(\cG_{\fP}, \bF_p^n)
}
\end{xy}
}
\noindent The lower map is surjective by Theorem \ref{thm:gruwi_cokerform}, and so is the upper.
\end{proof}

\begin{cor}
Let $K$ be number field, $S$ a set of primes of $K$. Let $K_S/L/K$ be a finite Galois subextension with Galois group $G$. Let $p$ be a prime and $A = \bF_p[G]^n$ a $\Gal_{K,S}$-module. Assume that $S$ is $p$-stable for $K_{S \cup S_p \cup S_{\infty}}/L$ with $p$-stabilizing field $L$. Then the embedding problem

\centerline{
\begin{xy}
\xymatrix{
& & & \Gal_{K,S} \ar@{->>}[d]& \\
1 \ar[r] & A \ar[r] & E \ar[r] & G \ar[r] & 1
}
\end{xy}
}

\noindent is properly solvable.
\end{cor}

\begin{proof}
It follows from Corollary \ref{cor:GWT_for_induced_module} in the same way as \cite{NSW} 9.2.9 follows from \cite{NSW} 9.2.4.
\end{proof}


\subsection{Grunwald-Wang cokernel in the limit and sharp $p$-stability} \label{sec:GWT_and_dagger}


If one is interested (motivated by Theorem \ref{thm:LocExt_RExt_CD}, we are) in the vanishing of the direct limit over $K_S/L/K$ of the Grunwald-Wang cokernel, rather than in the vanishing of the cokernel for each $L$, one can use sharp $p$-stability instead of strong $p$-stability, which is considerably weaker.

\begin{thm}\label{thm:dirlim_GW_coker_vanishes} Let $K$ be a number field, $S$ a set of primes of $K$ and $\cL \subseteq K_S$ a subextension normal over $K$, such that $S$ is sharply $p$-stable for $\cL/K$. Let $T$ be a finite set of primes of $K$ containing $(S_p \cup S_{\infty}) \sm S$. If $p^{\infty} | [\cL:K]$, then
\[ \dirlim_{\cL/L/K, \res} \coker^1(K_{S \cup T}/L, T, \bZ/p\bZ) = 0. \]
\end{thm}

\begin{proof}
For any finite subextension $\cL/L/K$ we have the short exact sequence

\[ 0 \rar \Sha^1(K_{S \cup T}/L, S \cup T; \mu_p) \rar \Sha^1(K_{S \cup T}/L, S \sm T; \mu_p) \rar \coker^1(K_{S \cup T}/L, T; \bZ/p\bZ)^{\vee} \rar 0. \]

\noindent Dualizing it, we see that it is enough to show that $\dirlim_{\cL/L/K, \cores^{\vee}} \Sha^1(K_{S \cup T}/L, S \sm T; \mu_p)^{\vee} = 0$. For any two finite subextensions $\cL/L^{\prime}/L/K$ we have the maps:
\begin{equation}\label{eq:resandcor}
\res_L^{L^{\prime}} \colon \Sha^1(K_{S \cup T}/L, S \sm T; \mu_p) \leftrightarrows \Sha^1(K_{S \cup T}/L^{\prime}, S \sm T; \mu_p) \colon \cores_L^{L^{\prime}}
\end{equation}

\begin{lm}\label{lm:resisbij}
There is a finite subextension $\cL/L_1/K$, such that for all $\cL/L^{\prime}/L/L_1$, the map $\res_L^{L^{\prime}}$ is an isomorphism.
\end{lm}
\begin{proof}
First we claim that $\res_L^{L^{\prime}}$ is injective if $L$ is big enough. Assume first that $\mu_p \subseteq \cL$ and $S$ is $p$-stable for $\cL/K$. Let $\cL/L_0/K$ be a finite subextension which $p$-stabilizes $S$ and contains $\mu_p$. Then any finite subextension $\cL/L/L_0$ satisfies the same. Assume $\res_L^{L^{\prime}}$ is not injective, i.e., there is some $0 \neq \phi \in \Sha^1(K_{S \cup T}/L, S \sm T; \bZ/p\bZ)$ with $\res_L^{L^{\prime}}(\phi) = 0$ (we have chosen some trivialization of $\mu_p$). This $\phi$ can be seen as a homomorphism $\phi \colon \Gal_{K_{S \cup T}/L} \rar \bZ/p\bZ$ which is trivial on all decomposition subgroups of primes in $S \sm T$. Let $M := (K_{S \cup T})^{\ker \phi}$. This is a finite Galois extension of $L$ with Galois group $\bZ/p\bZ$ and $\cs(M/L) \supseteq S \sm T$. But then
\[ \delta_M(S) = [M:L]\delta_L(S \cap \cs(M/L)) = p \delta_L(S), \]
\noindent since $T$ is finite. Now $\res_L^{L^{\prime}}(\phi) = 0$ implies $M \subseteq L^{\prime} \subseteq \cL$ and hence we get a contradiction to $p$-stability of $S$.

Now assume that $\mu_p \not\subseteq K_S$. Then $\res_L^{L^{\prime}}$ is always injective. Indeed, suppose there is an
\[ 0 \neq x \in \Sha^1(K_{S \cup T}/L, S \sm T; \mu_p) = \{ x \in L^{\ast}/p \colon x \in U_{\fp}L_{\fp}^{\ast,p} \text{ for } \fp \not\in S \cup T \text{ and } x \in L_{\fp}^{\ast,p} \text{ for } \fp \in S \sm T\} \]
\noindent with $\res_L^{L^{\prime}}(x) = 0$. This implies $x \in L^{\prime, \ast, p}$. Let $y^p = x$ with $y \in L^{\prime}$. Then $L(y) \subseteq L^{\prime} \subseteq \cL$. Since the polynomial $T^p - x$ is irreducible over $L$ (since $x \not\in L^{\ast, p}$), the conjugates of $y$ over $L$ are precisely the roots of this polynomial, which are obviously $\{ \zeta^i y \}_{i=0}^{p-1}$ for $\zeta \in \mu_p(\overline{K}) \sm \{1\}$. Since $\cL$ is normal over $L$, these conjugates lie in $\cL$. In particular, we deduce that $\zeta \in \cL$, which contradicts $\mu_p \not\subseteq \cL$. This finishes the proof of the injectivity claim.

By Corollary \ref{prop:uniformbound_forSha1_constcoeff}(ii), there is a constant $C > 0$ such that $\abs{\Sha^1(K_{S \cup T}/L, S \sm T, \mu_p)} < C$ for all $\cL/L/K$. Together with the injectivity shown above, this shows that there is a finite subextension $\cL/L_1/K$ such that for all $\cL/L^{\prime}/L/L_1$, the map $\res_L^{L^{\prime}}$ is bijective.
\end{proof}

Now we can finish the proof of Theorem \ref{thm:dirlim_GW_coker_vanishes}. Assume $L_1$ is as in Lemma \ref{lm:resisbij}. Let $\cL/L/L_1$. Since $p^{\infty}|[\cL:K]$, there is a further extension $\cL/L^{\prime}/L$ such that $p$ divides $[L^{\prime}:L]$. In the situation of \eqref{eq:resandcor} we have $\cores \circ \res = [L^{\prime}:L] = 0$ since $\mu_p$ is $p$-torsion. Dualizing gives $\res^{\vee} \circ \cores^{\vee} = (\cores \circ \res)^{\vee} = 0$. But with $\res$ also $\res^{\vee}$ is an isomorphism, hence we obtain $\cores^{\vee} = 0$. This shows
\[ \dirlim_{\cL/L/K, \cores^{\vee}} \Sha^1(K_{S \cup T}/L, S \sm T; \mu_p)^{\vee} = 0. \qedhere \]
\end{proof}

We also have the same arguments for $\Sha^2$.

\begin{prop} \label{prop:lim_over_Sha2_vanishes}
Let $K$ be a number field, $S$ a set of primes of $K$ and $\cL \subseteq K_S$ a subextension normal over $K$, such that $S$ is sharply $p$-stable for $\cL/K$. Let $T \supseteq S \cup S_p \cup S_{\infty}$ be a further set of primes. If $p^{\infty} | [\cL:K]$, then
\[ \dirlim_{\cL/L/K, \res} \Sha^2(K_T/L,T;\bZ/p\bZ) = 0. \]
\end{prop}

\begin{proof}
By Poitou-Tate duality this is equivalent to
\[ \dirlim_{\cL/L/K, \cores^{\vee}} \Sha^1(K_T/L, T; \mu_p)^{\vee} = 0. \]
\noindent This follows in the same way as in the proof of Theorem \ref{thm:dirlim_GW_coker_vanishes}.
\end{proof}


\subsection{Consequences}\label{sec:proof_of_dagger_results}

Here we prove Theorem \ref{thm:LocExt_RExt_CD}.

\begin{lm}\label{lm:pinfty_divides_ord_of_GalKS}
Let $S \supseteq R$ be sets of primes of $K$. Assume $R$ is finite and $S \cap \cs(K(\mu_p)/K)$ is infinite. Then $p^{\infty} | [K_{S}^{R}(p):K]$.
\end{lm}
\begin{proof} By \cite{NSW} 10.7.7, for any $C > 0$ there is some finite subset $S_C \subseteq S \cap \cs(K(\mu_p)/K)$ such that $R \subseteq S_C$ and
\[ \dim_{\bF_p} \coh^1(\Gal_{K,S_C}^{R}(p),\bZ/p\bZ) > C. \]
Since each group $\Gal_{K,S_C}^{R}(p)$ is a quotient of $\Gal_{K,S}^{R}(p)$, the lemma follows.
\end{proof}

\begin{proof}[Proof of Theorem \ref{thm:LocExt_RExt_CD}]
($A_p$): Let $\fp$ be a prime of $K$, which is not contained in $R$. Since the local group $\cG_{\fp}(p)$ is solvable and the assumptions carry over to extensions of $K$ in $K_S^R(p)$, it is enough to show that any class $\alpha_{\fp} \in \coh^1(\cG_{\fp}(p),\bZ/p\bZ)$ (which has to be unramified if $\fp \not\in S$) is realized by a global class after a finite extension. Let $T := \{ \fp \} \cup R \cup S_p \cup S_{\infty}$ and let $(\alpha_{\fq}) \in \prod_{\fq \in T} \coh^1(\cG_{\fp}(p), \bZ/p\bZ)$, such that $\alpha_{\fq}$ is unramified if $\fq \not\in S$ and $0$ if $\fp \in R$. By Theorem \ref{thm:dirlim_GW_coker_vanishes}, there is some finite extension $K_S^R(p)/L/K$, such that $(\alpha_{\fq})$ comes from a global class $\alpha \in \coh^1(\Gal_{L,S \cup T}^R(p), \bZ/p\bZ)$. The $\bZ/p\bZ$-extension of $L$ corresponding to $\alpha$ is unramified outside $S$, completely split in $R$ and hence contained in $K_S^R(p)$, which finishes the proof. (A) has analogous proof.

($B_p$): The proof essentially coincides with the proofs of \cite{NSW} 10.5.8 resp. \cite{Iv} Theorem 4.26. As done there, we can restrict ourselves to the case $T \supseteq S_p \cup S_{\infty}$. All cohomology groups in the proof have $\bZ/p\bZ$-coefficients and we omit them from the notation. After computing the cohomology on the left side, by \cite{NSW} 1.6.15 we have to show that the map
\[ \coh^i(\phi_{T,S}^R(p)) \colon \coh^i( K_T(p)/K_S^R(p)) \rar \Resdirsum_{\fp \in R(K_S^R(p))} \coh^i(\cG_{\fp}(p)) \oplus \Resdirsum_{\fp \in (T \sm S)(K_S^R(p))} \coh^i(I_{\fp}(p)) \]

\noindent induced by $\phi_{T,S}^R(p)$ in the cohomology is bijective for $i = 1$ and injective for $i = 2$ (here $\Resdirsum$ means the restricted direct sum in the sense of \cite{NSW} 4.3.13). Now, $\coh^1(\phi_{T,S}^R(p))$ is injective since $\phi_{T,S}^R(p)$ is clearly surjective. To show surjectivity for $i = 1$, consider for any finite subset $T_1 \subseteq T \sm S$, which contains $(S_p \cup S_{\infty}) \sm S$, and any finite $K_S^R(p)/L/K$ the composed maps:

\[ \coh^1(K_{S \cup T_1}(p)/L) \rar \bigoplus_{\fp \in (R \cup T_1)(L)} \coh^1(\cG_{\fp}) \tar \bigoplus_{\fp \in R(L)} \coh^1(\cG_{\fp}) \oplus \bigoplus_{\fp \in T_1(L)} \coh^1(\cI_{\fp})^{\cG_{\fp}}, \]

\noindent where $\cI_{\fp} = I_{\overline{K_{\fp}}/L_{\fp}} \subseteq  \Gal_{\overline{K_{\fp}}/L_{\fp}} = \cG_{\fp}$ is the inertia subgroup. Passing to the direct limit over $K_S^R(p)/L/K$, we obtain by Theorem \ref{thm:dirlim_GW_coker_vanishes} the surjection
\[ \coh^1(K_{S \cup T_1}(p)/K_S^R(p)) \tar \Resdirsum_{\fp \in R(K_S^R(p))} \coh^1(\cG_{\fp}(p)) \oplus \Resdirsum_{\fp \in T_1(K_S^R(p))} \coh^1(I_{\overline{K_{\fp}}/K_{\fp}})^{ \Gal_{\overline{K_{\fp}}/K_{S,\fp}^R(p) } }, \]

\noindent which is, after passing to the direct limit over all finite $T_1 \subseteq T \sm S$, exactly $\coh^1(\phi_{T,S}^R(p))$, since by $(A_p)$ we have $K_S^R(p)_{\fp} = K_{\fp}^{\nr}(p)$ for $\fp \in T \sm S$ and hence $\coh^1(I_{\overline{K_{\fp}}/K_{\fp}})^{ \Gal_{\overline{K_{\fp}}/K_{S,\fp}^R(p) } } = \coh^1(I_{\fp}(p))$ (cf. the proofs of \cite{NSW} 10.5.8 resp. \cite{Iv} 4.26). Finally, the  injectivity of $\coh^2(\phi_{T,S}^R(p))$ follows by passing to the limit and using Proposition \ref{prop:lim_over_Sha2_vanishes}.
(B): By Lemma \ref{lm:dagger_rel_p_is_nice}, there is some $K_S^R/L_0/K$, such that for all $K_S^R/L/L_0$, $S$ is sharply $p$-stable for $L_S^R(p)/L$. Thus (B) follows from ($B_p$) as we have
\[ I_{\fp}^{\prime}(p) = \prolim_{K_S^R/L/K} I_{L_{\fp}(p)/L_{\fp}} \]
\noindent and
\[ \Gal_{K^{\prime}_T (p)/K_S^R} = \prolim_{K_S^R/L/K} \Gal_{L_T(p)/L_S^R(p)}. \]

($C_p$),(C): The proof essentially coincides with the proofs of \cite{NSW} 10.5.10, 10.5.11 resp. \cite{Iv} Theorem 4.31, Corollary 4.33. To avoid many repetitions, we only recall the argument for $\cd \Gal_{K,S}^R(p) \leq 2$ in the case $R = \emptyset$ (which differs in one aspect from the cited proofs). Therefore, let $V = (S_p \cup S_{\infty}) \sm S$ and consider the Hochschild-Serre spectral sequence $(E_n^{ij}, \delta_n^{ij})$ for the Galois groups of the global extensions
$K_{S \cup V}(p)/K_S(p)/K$. By \cite{NSW} 8.3.18 and 10.4.8, we have:
\[ \cd \Gal_{K, S \cup V}(p) \leq \cd_p \Gal_{K,S \cup V} \leq 2. \]

\noindent By Riemann's existence theorem ($B_p$) the group $\Gal_{K_{S \cup V}(p)/K_S(p)}$ is a free pro-$p$ group. Hence $E_n^{ij}$ degenerates in the second tableau and in particular, we have (omitting $\bZ/p\bZ$-coefficients from the notation)
\[ \coker(\delta_2^{11}) = E_3^{30} = E_{\infty}^{30} \subseteq \coh^3(\Gal_{K, S \cup V}(p)) = 0.\]

\noindent I.e., $\delta_2^{11}$ is surjective. Again by Riemann's existence theorem we have
\begin{equation*}
\coh^1(K_{S \cup V}(p)/K_S(p) ) \cong \bigoplus_{\fp \in V} \Ind\nolimits_{D_{\fp, K_S(p)/K} }^{\Gal_{K,S}(p)} \coh^1(I_{\fp}(p)),
\end{equation*}
\noindent This and Shapiro's lemma imply
\begin{equation}\label{eq:E211comp_for_cd2_sspaper} E_2^{11} = \bigoplus_{\fp \in V} \coh^2(K_{\fp}(p)/K_{\fp}).
\end{equation}

\noindent Further, we have the following commutative diagram with exact rows and columns:

\centerline{
\begin{xy}\label{diag:hihi`}
\xymatrix{
& \bigoplus_{\fp \in S} \coh^2(K_{\fp}(p)/K_{\fp}) \ar@{^{(}->}[d] \ar@{->>}[r] & \coh^0(K_{S \cup V}/K, \mu_p)^{\vee} \ar[d]^{=} \\
\coh^2(K_{S \cup V}(p)/K) \ar[r] \ar[d] & \bigoplus_{\fp \in S \cup V} \coh^2(K_{\fp}(p)/K_{\fp}) \ar@{->>}[r] \ar@{->>}[d] & \coh^0(K_{S \cup V}/K, \mu_p)^{\vee} \ar[d] \\
\coh^1(K_S(p)/K, \coh^1(K_{S \cup V}(p)/K_S(p))) \ar[r]^(0.6){\sim} \ar@{->>}[d]^{\delta_2^{11}} & \bigoplus_{\fp \in V} \coh^2(K_{\fp}(p)/K_{\fp}) \ar[r] & 0 \\
\coh^3(K_S(p)/K) & &
}
\end{xy}
}

\noindent in which the second row comes from the Poitou-Tate long exact sequence. The first map in the third row is the isomorphism \eqref{eq:E211comp_for_cd2_sspaper}. The map in the first row is surjective since its dual map $\mu_p(K) \rar \oplus_{\fp \in S} \mu_p(K_{\fp})$ is injective. Now (in contrast to proofs cited from \cite{NSW} and \cite{Iv}) the first map in the second row is not necessarily injective, but one can simply replace the first entry in the second row by ${\rm H}^2(K_{S \cup V}(p)/K)/ \Sha^2(K_{S \cup V}/K, S \cup V; \bZ/p\bZ)$, as both maps in the diagram which start at this entry factor through this quotient. Now apply the snake lemma to the second and the third row and obtain $\coh^3(K_S(p)/K) = 0$ and hence also $\cd \Gal_{K,S}(p) \leq 2$ by \cite{NSW} 3.3.2. \qedhere
\end{proof}


\subsection{Vanishing of $\Sha^2(G_S; \bZ/p\bZ)$ without $p \in \caO_{K,S}^{\ast}$} \label{sec:hasse_principle_for_Sha2_without_p_inv}

We generalize Corollary \ref{cor:Sha2is0forstable} for the constant module. The proof makes use of Theorem \ref{thm:LocExt_RExt_CD} (A), (B), (C) along with the result of Neumann showing the vanishing of certain cohomology groups. Its special case $\delta_K(S) = 1$ is not contained in \cite{NSW}. Part (i) is \cite{Iv} 4.34.

\begin{prop} \label{prop:Sha2vanishingwoinverting}
Let $K$ be a number field, $S$ a set of primes of $K$. Let $p$ be a rational prime, $r > 0$ an integer. Assume that either $p$ is odd, or $K_S$ is totally imaginary. Then the following holds:
\begin{itemize}
\item[(i)]  Assume $S$ is strongly $p$-stable and $L_0$ is a $p$-stabilizing field for $S$ for $K_{S \cup S_p \cup S_{\infty}}/K$. Assume $p$ is odd, or $L_0$ is totally imaginary. Then
\[ \Sha^2(K_S/L; \bZ/p^r\bZ) = 0 \]
\noindent for any finite $K_S/L/L_0$, such that we are not in the special case $(L, p^r, S)$.
\item[(ii)] Let $K_S/\cL/K$ be a normal subextension. Assume that $S$ is sharply $p$-stable for $\cL/K$ and $p^{\infty}|[\cL:K]$. Then
\[ \dirlim_{\cL/L/K} \Sha^2(K_S/L; \bZ/p^r\bZ) = 0. \]
\end{itemize}
\end{prop}


\begin{proof}
Let $V := (S_p \cup S_{\infty}) \sm S$. In the following, we write $\coh^{\ast}(\cdot)$ instead of  $\coh^{\ast}(\cdot, \bZ/p^r\bZ)$ and $\Sha^{\ast}(\cdot,\cdot)$ instead of $\Sha^{\ast}(\cdot,\cdot; \bZ/p^r\bZ)$. Let $K_{S \cup V}^{\prime}(p)$ be the maximal pro-$p$-subextension of $K_{S \cup V}/K_S$. Let $K_S/L/K$ be a finite subextension and consider the following tower of extensions:

\centerline{
\begin{xy}
\xymatrix{
K_{S \cup V} \ar@{-}[d]^N \ar@{-} `l[d] `d[ddd]_{\Gal_{L,S \cup V}} \\
K_{S \cup V}^{\prime}(p) \ar@{-}[d]^H \ar@{-} `r[d] `d[dd]^{\Gal_{L,S \cup V}^{\prime}(p)} \\
K_S \ar@{-}[d]^{\Gal_{L,S}} \\
L
}
\end{xy}
}

\noindent with $N := \Gal_{K_{S \cup V}/ K_{S \cup V}^{\prime}(p)}$, $H := \Gal_{K_{S \cup V}^{\prime}(p)/K_S}$ and $\Gal_{L,S \cup V}^{\prime}(p) := \Gal_{K_{S \cup V}^{\prime}(p)/L}$. We claim that for any such $L$ we have under the assumptions of (i) resp. (ii), the following natural isomorphisms:

\begin{align}\label{eqn:iso_iso_von_Sha_2}
& \Sha^2(K_{S \cup V}^{\prime}(p)/L, S \cup V) = \Sha^2(K_{S \cup V}/L, S \cup V)  \quad \text{ for any $K_S/L/K$ and } \nonumber \\
& \Sha^2(K_S/L, S) = \Sha^2(K_{S \cup V}^{\prime}(p)/L, S \cup V) \quad \text{for any $K_S/L/L_0$ under (i), and} \\
& \dirlim_{\cL/L/K} \Sha^2(K_S/L, S) = \dirlim_{\cL/L/K}  \Sha^2(K_{S \cup V}^{\prime}(p)/L, S \cup V) \quad \text{under (ii).} \nonumber
\end{align}

\noindent Once this claim is shown, (i) follows immediately from Corollary \ref{cor:Sha2is0forstable} and (ii) follows from Proposition \ref{prop:lim_over_Sha2_vanishes}. Thus it is enough to prove the above claim. The first isomorphism in \eqref{eqn:iso_iso_von_Sha_2} follows immediately from the definition of $\Sha^2$, once we know that the inflation map $\coh^2(\Gal_{L, S \cup V}^{\prime}(p)) \rar \coh^2(\Gal_{L, S \cup V})$ is an isomorphism. To show this last assertion, consider the Hochschild-Serre spectral sequence

\[ E^{ij}_2 = \coh^i(\Gal_{L, S \cup V}^{\prime}(p), \coh^j(N)) \Rar \coh^{i+j}(\Gal_{L, S \cup V}). \]

\noindent A result of Neumann (\cite{NSW} 10.4.2) applied to $K_{S \cup V}/K_{S \cup V}^{\prime}(p)$ (the upper field is $p-(S\cup V)$-closed, the lower is $p-(S_p \cup S_{\infty})$-closed) implies $E^{ij}_2 = 0$ for $j > 0$, hence the sequence degenerates in the second tableau and
\[ \coh^i(\Gal_{S \cup V}^{\prime}(p)) = \coh^i(\Gal_{S \cup V}), \]
\noindent for $i \geq 0$, proving our claim. Thus we are reduced to show that the second and the third maps in \eqref{eqn:iso_iso_von_Sha_2} are isomorphisms. For $\fp \in V$, let $K_{\fp}^{\prime}(p)$ denote the maximal pro-$p$ extension of $K_{S,\fp}$. Let
\[ I_{\fp}^{\prime}(p) :=  \Gal_{K_{\fp}^{\prime}(p)/K_{S,\fp}} \]

\noindent (observe that if $\fp \in S_{\infty}$, then $I_{\fp}^{\prime}(p) = 1$. Indeed, if $p > 2$, this is always the case, and if $p = 2$, then $K_{S,\fp} = \bC$ using the assumption that $K_S$ is totally imaginary). By \cite{Iv} Lemma 4.23 (which was only shown there under strong $p$-stability assumption on $S$, but due to Theorem \ref{thm:LocExt_RExt_CD} (A), it also holds under sharp $p$-stability assumption with exactly the same proof), we have $I_{\fp}^{\prime}(p) = D_{\fp, K_{S \cup V}^{\prime}(p)/K_S}$. By Riemann's existence Theorem \ref{thm:LocExt_RExt_CD} (B) applied to $K_{S \cup V}^{\prime}(p)/K_S/K$, we have
\[ H \cong \bigast\limits_{\fp \in V(K_S)} I_{\fp}^{\prime}(p). \]

\noindent By \cite{Iv} Corollary 4.24, the groups $I_{\fp}^{\prime}(p)$ are free pro-$p$-groups, and hence $H$ is free pro-$p$-group. Thus $\cd_p H \leq 1$. Consider the exact sequence
\[1 \rar H \rar \Gal_{L, S \cup V}^{\prime}(p) \rar \Gal_{L,S} \rar 1, \]

\noindent and the corresponding Hochschild-Serre spectral sequence
\[ E_2^{ij} = \coh^i(\Gal_{L,S}, \coh^j(H)) \Rightarrow \coh^{i+j}(\Gal_{L,S \cup V}^{\prime}(p))\]

\noindent Since by Theorem \ref{thm:LocExt_RExt_CD}(C) we know that $\cd_p \Gal_{L,S} = 2$, we have $E^{ij}_2 = 0$ if $i > 2$ or $j > 1$. Moreover, we have

\[ \coh^1(H) = \Resdirsum_{V(K_S)} \coh^1(I_{\fp}^{\prime}(p)) = \bigoplus_{V(L)} \Ind\nolimits_{D_{\fp, K_S/L}}^{\Gal_{L,S}} \coh^1(I_{\fp}^{\prime}(p)) \]

\noindent as $\Gal_{L,S}$-modules, where $D_{\fp,K_S/L} \subseteq \Gal_{L,S}$ is the decomposition group at $\fp$, which is in particular pro-cyclic and has an infinite $p$-Sylow subgroup (by Theorem \ref{thm:LocExt_RExt_CD}(A)). Using this, an easy computation involving Frobenius reciprocity, Shapiro's lemma and \cite{Iv} Lemma 4.24 allows us to compute the terms $E_2^{01}$ and $E_2^{11}$. We obtain the following exact sequence (where $\delta := \delta_2^{01} \colon E_2^{01} \rar E_2^{20}$ denotes the differential in the second tableau):

\centerline{
\begin{xy}
\xymatrix{
0 \ar[r] & \coh^1(\Gal_{L,S}) \ar[r] & \coh^1(\Gal_{L,S \cup V}^{\prime}(p)) \ar[r] & \bigoplus\limits_{V(L)} \coh^1(I_{\fp}^{\prime}(p))^{D_{\fp, K_S/L}} \ar[r] & \\
\ar[r]^(0.3){\delta} & \coh^2(\Gal_{L,S}) \ar[r] & \coh^2(\Gal_{L,S \cup V}^{\prime}(p)) \ar^(0.5){d}[r] & \bigoplus\limits_{V(L)} \coh^2(\cG_{\fp}) \ar[r] & 0.
}
\end{xy}
}
\noindent Assume first we are in the situation of (i) and let $L$ be as introduced there. We have the following surjections:

\[ \coh^1(\Gal_{L,S \cup V}^{\prime}(p)) \tar \bigoplus_{\fp \in V(L)} \coh^1(\cG_{\fp}) = \bigoplus_{\fp \in V(L)} \coh^1(D_{\fp,K_{S \cup V}^{\prime}(p)/L}) \tar \bigoplus_{V(L)} \coh^1(I_{\fp}^{\prime}(p))^{D_{\fp,K_S/L}}, \]

\noindent (the first map is surjective by Grunwald-Wang Theorem \ref{thm:gruwi_cokerform}, and the second and the third maps follow from \cite{Iv} Lemma 4.24. Hence the map preceding $\delta$ is surjective and hence $\delta = 0$. Thus the lower row of the above 6-term exact sequence gives the short exact sequence

\vspace{1ex}

\centerline{
\begin{xy}
\xymatrix{
0 \ar[r] & \Sha^2(K_S/L, S) \ar[r] & \Sha^2(K_{S \cup V}^{\prime}(p)/L, S) \ar^(0.55){d}[r] & \bigoplus\limits_{V(L)} \coh^2(\cG_{\fp}),
}
\end{xy}
}

\noindent On the other side, by definition of $\Sha^2$, the kernel of $d$ is precisely $\Sha^2(K_{S \cup V}^{\prime}(p)/K, S \cup V)$, which shows the second equality in \eqref{eqn:iso_iso_von_Sha_2}. The third equality in \eqref{eqn:iso_iso_von_Sha_2} follows from the assumptions in (ii) by the same arguments after taking $\dirlim$ over $\cL/L/K$ (and using Theorem \ref{thm:dirlim_GW_coker_vanishes} instead of Theorem \ref{thm:gruwi_cokerform}).
\end{proof}




\section{$\Kapi$-property} \label{sec:Kapi_section}

Assume that either $p$ is odd, or $K$ is totally imaginary and let
\[X = \Spec \caO_{K,S}. \]

\noindent While it is well known that $X$ is a $\Kapi$ for $p$ if either $S \supseteq S_p \cup S_{\infty}$ (``wild case''), or $\delta_K(S) = 1$, it is a challenging problem to determine whether $X$ is a $\Kapi$ if $S$ is finite and not necessarily contains $S_p \cup S_{\infty}$. Until recently there were no non-trivial examples of $(K,S)$ such that $X$ is a $\Kapi$ for $p$ or a pro-$p$ $\Kapi$ and, say, $S \cap S_p = \emptyset$. Recent results of A. Schmidt (\cite{Sch}, \cite{Sch2}, cf. also \cite{Sch3}) show that any point of $\Spec\caO_K$ has a basis for Zariski-topology consisting of pro-$p$ $\Kapi$-schemes. More precisely, given $K$, a finite set $S$ of primes of $K$, a rational prime $p$ and any set $T$ of primes of $K$ of density 1, Schmidt showed that one can find a finite subset $T_1 \subseteq T$ such that $X \sm T_1$ is pro-$p$ $\Kapi$. The main ingredient in the proof is the theory of mild pro-$p$ groups, developed by Labute. We conjecture that one can replace the condition $\delta_K(T) = 1$ in 
Schmidt's work by the weaker condition that $T$ is strongly $p$-stable (or even that $T$ is sharply $p$-stable for $K_T(p)/K$).

In the present section we enlarge the set of the examples of such pairs $(K,S)$, for which $X$ is a $\Kapi$ for $p$ and prove essentially that if $S$ is sharply  $p$-stable, then $X$ is a $\Kapi$ for $p$. In particular, if $S$ is a stable almost Chebotarev set with $S_{\infty} \subseteq S$, then $X$ is a $\Kapi$ for almost all primes $p$ (cf. Proposition \ref{prop:finitnessofES} and Example \ref{ex:pers_ES_fin_and_nonempty}), and if $E^{\rm sharp}(S) = \emptyset$ and $K$ is totally imaginary, then $X$ is a $\Kapi$.

\subsection{Generalities on the $\Kapi$-property}

There are many equivalent ways to characterize the $\Kapi$-property of schemes (cf. \cite{St} Appendix A, where they are discussed in detail). Without repeating all of them, we want to introduce a small refinement of terminology, such that it is better adapted to formulate our results.

To begin with, let $X$ be a connected scheme, $X_{\et}$ the \'{e}tale site on $X$. Fix a geometric point $\bar{x} \in X$ and let $\pi := \pi_1(X,\bar{x})$ be the \'{e}tale fundamental group of $X$. Let $\cB \pi$ denote the site of continuous $\pi$-sets endowed with the canonical topology. Let further $p$ be a rational prime, and let $\cB \pi^p$ denote the site of continuous $\pi^{(p)}$-sets, where $\pi^{(p)}$ is the pro-$p$ completion of $\pi$. As in \cite{St} A.1, we have natural continuous maps of sites

\centerline{
\begin{xy}
\xymatrix{
X_{\et} \ar[r]^{\gamma} \ar[dr]_{\gamma_p} & \cB \pi \ar[d] \\
& \cB \pi^p
}
\end{xy}
}

\noindent For a site $Y$, let $\cS(Y)$ denote the category of sheaves of abelian groups on $Y$, let $\cS(Y)_f$ be the subcategory of locally constant torsion sheaves, and $\cS(Y)_p$ the subcategory of locally constant $p$-primary torsion sheaves. Let $A \in \cS(\cB\pi)_f$ resp. $B \in \cS(\cB\pi^p)_p$. Then we have the natural transformations of functors $\id \rar \R \gamma_{\ast} \gamma^{\ast}$ resp. $\id \rar \R \gamma_{p,\ast} \gamma_p^{\ast}$, which induce maps in the cohomology:

\begin{eqnarray*}
c_A^i \colon& \coh^i(\pi, A) &\longrightarrow \coh^i(X_{\et}, \gamma^{\ast} A ) \\
c_{p,B}^i \colon& \coh^i(\pi^{(p)}, B) &\longrightarrow \coh^i(X_{\et}, \gamma_p^{\ast} B).
\end{eqnarray*}

\noindent Let $\tilde{X}$ resp. $\tilde{X}(p)$ denote the universal resp. the universal pro-$p$ covering of $X$. Since
\[ \coh^1(\tilde{X}_{\rm et}, A) = \coh^1(\tilde{X}(p)_{\rm et}, B) = 0 \]

\noindent for each $A, B$, the maps $c_A^i$ and $c_{p,B}^i$ are isomorphisms for $i = 0,1$ and are injective for $i = 2$.

\begin{Def}\label{def:Kapidef} Let $X$ be a connected scheme.
\begin{itemize}
\item[(i)]   $X$ is \emph{a $\Kapi$} if $c_A^i$ is an isomorphism for all $A \in \cS(\cB\pi)_f$ for all $i \geq 0$.
\item[(ii)]  $X$ is \emph{a $\Kapi$ for $p$} if $c_A^i$ is an isomorphism for all $A \in \cS(\cB\pi)_p$ for all $i \geq 0$.
\item[(iii)] $X$ is \emph{a pro-$p$ $\Kapi$} if $c_{p,B}^i$ is an isomorphism for all $B \in \cS(\cB\pi^p)_p$ for all $i \geq 0$.
\end{itemize}
\end{Def}

Notice that we use a shift in the definitions compared with \cite{Sch} or \cite{Wi2}: what there is called a $\Kapi$ for $p$, we call here a pro-$p$ $\Kapi$. Parts (i) and (iii) of our definition coincide with the definition of a $K(\pi,1)$ in \cite{St} A.1.2. By decomposing any sheaf into $p$-primary components we obtain:

\begin{lm}
$X$ is a $\Kapi$ if and only if it is a $\Kapi$ for all $p$.
\end{lm}

Now we have a criterion for being $\Kapi$. For a scheme $X$ let $\Fet_X$ (resp. $\Fet_X^{(p)}$) denote the category of all finite \'{e}tale coverings (resp. finite \'etale $p$-coverings) of $X$. For a number field $K$ let

\[ \delta_K = \begin{cases} 1 & \text{if } \mu_p \subseteq K, \\ 0 & \text{otherwise.} \end{cases}\]

\begin{prop}\label{prop:arithKp1_App}
Let $K$ be a number field, $S \supseteq S_{\infty}$ a set of primes of $K$ such that either $\delta_K = 0$ or $S_f \neq \emptyset$. Assume that either $p$ is odd, or $K$ is totally imaginary. Let $X = \Spec \caO_{K,S}$. The following are equivalent:
\begin{itemize}
\item[(i)] $X$ is a $K(\pi,1)$ for $p$.
\item[(ii)] One has
\[ \dirlim_{Y \in \Fet_X } \coh^2(Y_{\et},\bZ/p\bZ) = 0. \]
\end{itemize}
The same also holds if one replaces '$K(\pi,1)$ for $p$' by 'pro-$p$ $\Kapi$' and '$\Fet_X$' by '$\Fet_X^{(p)}$' respectively.
\end{prop}

\begin{proof} For the full proof, cf. \cite{Iv} Proposition 5.5. For convenience, we sketch here the main steps. (i) $\Rightarrow$ (ii) holds for any connected scheme and follows from \cite{St} A.3.1 and (ii) $\Rightarrow$ (i) follows from the well-known criterion \cite{St} A.3.1, and the fact that for every $q > 0$ and every locally constant $p$-primary torsion sheaf $A$ on $X_{\et}$, we have

\[ \dirlim_{Y \in \Fet_X} \coh^q(Y_{\et},A|_Y) = 0. \]

\noindent Since $A$ is trivialized on some $Y \in \Fet_X$, we can assume that $A$ is constant. By d\'{e}vissage we are reduced to the case $A = \bZ/p\bZ$. The elements of $\coh^1(Y_{\et},\bZ/p\bZ)$ can be interpreted as torsors, which kill themselves, i.e., the case $q = 1$ follows. Further by \cite{SGA4} Expos\'{e} X Proposition 6.1, $\coh^q(Y_{\et},\bZ/p\bZ) = 0$ for $q > 3$. The case $q = 3$ follows from Artin-Verdier duality. Finally, (ii) implies the case $q = 2$. The pro-$p$ case has a similar proof.
\end{proof}


\subsection{$\Kapi$ and sharp $p$-stability}

\begin{thm}\label{thm:Kapi1forstable}
Let $K$ be a number field, $S \supseteq S_{\infty}$ a set of primes of $K$ and $p$ a rational prime. Assume that either $p$ is odd, or $K$ is totally imaginary. The following holds:
\begin{itemize}
\item[(i)] If $S$ is sharply $p$-stable for $K_S(p)/K$, then $\Spec \caO_{K,S}$ is a pro-$p$ $\Kapi$. \\
\item[(ii)] If $S$ is sharply $p$-stable, then $\Spec \caO_{K,S}$ is a $\Kapi$ for $p$.

\end{itemize}
\end{thm}

\begin{rem}
If $K$ is totally imaginary or in the pro-$p$ case, the assumption $S_{\infty} \subseteq S$ is superfluous as $\Gal_S(p) = \Gal_{S \cup S_{\infty}}(p)$: if $p > 2$, then this is true in general and if $p = 2$, then this is true since we have assumed that $K$ is totally imaginary.
\end{rem}

\begin{cor}\label{cor:Kapi_corollar}
Let $K$ be a number field, $S \supseteq S_{\infty}$ a stable set of primes of $K$, such that $E^{\rm sharp}(S)$ is finite (in particular, $S$ can be any stable almost Chebotarev set with $S \supseteq S_{\infty}$). Then $\Spec \caO_{K,S}$ is a $K(\pi,1)$ for almost all primes $p$. If $E^{\rm sharp}(S) = \emptyset$ and $K$ is totally imaginary, then $\Spec \caO_{K,S}$ is a $K(\pi,1)$.
\end{cor}

\begin{ex} Let $K$ be totally imaginary. Let $\tilde{K} := \bigcup_p K(\mu_p)$. Let $M/K$ be finite Galois with $M \cap \tilde{K} = K$ and $\sigma \in \Gal_{M/K}$. Assume that $S \backsimeq P_{M/K}(\sigma)$ is stable. Then $\Spec \caO_{K,S}$ is a $\Kapi$.
\end{ex}

\begin{proof}[Proof of Theorem \ref{thm:Kapi1forstable}] (The proof essentially coincides with that of \cite{Iv} Theorem 5.12) We only prove (ii) (the pro-$p$ case (i) has a similar proof). Let $X := \Spec \caO_{K,S}$. As $L$ goes through finite subextensions of $K_S/K$, the normalization $Y$ of $X$ in $L$ goes through all finite  \'{e}tale connected coverings of $X$.
Let $V := S_p \sm S$. For any such $Y$ we have a decomposition
\[ Y \sm V \stackrel{j}{\har} Y \stackrel{i}{\hookleftarrow} V \]

\noindent in an open and a closed part. Now $Y \sm V$ is a $\Kapi$ for $p$ and $\pi_1(Y \sm V) = \Gal_{L,S \cup V}$. Hence
\begin{equation} \label{eq:XsmVist_Kapi}
c_A^i \colon \coh^i(\Gal_{L,S \cup V}) \stackrel{\sim}{\longrightarrow} \coh^i((Y \sm V)_{\et}, A)
\end{equation}

\noindent is an isomorphism for any $i \geq 0$ and any $p$-primary $\Gal_{L,S \cup V}$-module $A$. We have the Lerray spectral sequence for $j$:
\[ E^{mn}_2 = \coh^m(Y, R^n j_{\ast} \bZ/p\bZ) \Rar \coh^{m + n}(Y \sm V, \bZ/p\bZ). \]

\noindent Let us compute the terms in this spectral sequence. First of all we have
\[ R^n j_{\ast} \bZ/p\bZ = \begin{cases} \bZ/p\bZ & \text{if } n = 0, \\ \bigoplus_{\fp \in V} \coh^1(\cI_{\fp}, \bZ/p\bZ) & \text{if } n = 1, \\ 0 & \text{if } n > 1, \end{cases} \]

\noindent where $\cI_{\fp} \subseteq \cG_{\fp}$ denotes the inertia subgroup of the full local Galois group at $\fp$. Thus
\begin{eqnarray*}
E^{01}_2 &=& \bigoplus_{\fp \in V} \coh^1(\cI_{\fp}, \bZ/p\bZ)^{\cG_{\fp}^{\nr}} \\
E^{11}_2 &=& \coh^1(Y_{\et}, \bigoplus_{\fp \in V} \coh^1(\cI_{\fp}, \bZ/p\bZ)) = \bigoplus_{\fp \in V} \coh^2(\cG_{\fp}, \bZ/p\bZ) \\
\end{eqnarray*}

\noindent and $E^{mn}_2 = 0$ if $n > 1$ or if $n = 1$ and $m > 1$ (as $\cd_p(\cG_{\fp}^{\nr}) = 1$). Further, $E^{m0}_2 = 0$ for $m > 3$, as $\cd_p Y \leq 3$ and $E^{30}_2 = \coh^3(Y, \bZ/p\bZ) = 0$ by \cite{Iv} Lemma 5.9. Further,
\[ E^{10}_2 = \coh^1(Y_{\et},\bZ/p\bZ) = \coh^1(\Gal_{L,S}, \bZ/p\bZ). \]

\noindent Thus we have the following non-zero entries in the second tableau:

\[ \xymatrix{
\bigoplus_{\fp \in V} \coh^1(\cI_{\fp}, \bZ/p\bZ)^{\cG_{\fp}^{\nr}} \ar[rrd]^{\delta_2^{01}}	& \bigoplus_{\fp \in V} \coh^2(\cG_{\fp}, \bZ/p\bZ) & 0 & 0 \\
\bZ/p\bZ					& \coh^1(\Gal_{L,S}, \bZ/p\bZ)				& \coh^2(Y_{\et}, \bZ/p\bZ) & 0
}
\]

\noindent From this and the isomorphism \eqref{eq:XsmVist_Kapi} we obtain the following exact sequence (from now on, we omit the $\bZ/p\bZ$-coefficients):

\[ \xymatrix{
0 \ar[r] & \coh^1(\Gal_{L,S}) \ar[r] 	& \coh^1(\Gal_{L,S \cup V}) \ar[r] & \bigoplus_{\fp \in V} \coh^1(\cI_{\fp})^{\cG_{\fp}^{\nr}} \ar[r]^(0.75){\delta_2^{01}} & \\
  \ar[r] & \coh^2(Y_{\et}) \ar[r]		& \coh^2(\Gal_{L,S \cup V}) \ar[r] & \bigoplus_{\fp \in V} \coh^2(\cG_{\fp}) \ar[r] & 0
}
\]

\noindent By Proposition \ref{prop:arithKp1_App} it is enough to show that $\dirlim_{Y \in \Fet_X} \coh^2(Y_{\et}) = 0$. Taking the limit over all $Y \in \Fet_X$ of this sequence, we see by Theorem \ref{thm:dirlim_GW_coker_vanishes} that the direct limit of the maps preceding $\delta_2^{01}$ is surjective, hence we obtain:

\[ \dirlim_{Y \in \Fet_X} {\rm H}^2(Y_{\et}) \cong \dirlim_{Y \in \Fet_X} \Sha^2(K_{S \cup V}/L, V; \bZ/p\bZ). \]


\noindent To finish the proof consider the following commutative diagram with exact rows:

\[ \xymatrix{
         & \coh^2(\Gal_{L, S \cup V}) \ar[r] \ar[d] & \bigoplus_{\fp \in S \cup V} \coh^2(\cG_{\fp}) \ar[r] \ar[d] & \mu_p(L)^{\vee} \ar[r] \ar[d] & 0 \\
0 \ar[r] &  \bigoplus_{\fp \in V} \coh^2(\cG_{\fp}) \ar[r]^{=} &  \bigoplus_{\fp \in V} \coh^2(\cG_{\fp}) \ar[r] & 0 \ar[r] & 0,
}
\]

\noindent in which the first map in the upper row gets injective after taking the limit by Proposition \ref{prop:lim_over_Sha2_vanishes}. Snake lemma shows that

\[ \dirlim_{Y \in \Fet_X} {\rm H}^2(Y_{\et}) \cong \dirlim_{Y \in \Fet_X} \Sha^2(K_{S \cup V}/L, V; \bZ/p\bZ) \subseteq \dirlim_{Y \in \Fet_X} \bigoplus_{\fp \in S} \coh^2(\cG_{\fp}), \]

\noindent and the last limit vanishes as $p^{\infty}| [K_{S,\fp}:K_{\fp}]$ for all $\fp \in S$ by Theorem \ref{thm:LocExt_RExt_CD}(A). This finishes the proof of (ii).
\end{proof}


\renewcommand{\refname}{References}

\end{document}